\documentclass[preprint]{imsart}

\RequirePackage{amsthm,amsmath,amsfonts,amssymb}
\RequirePackage[numbers]{natbib}
\RequirePackage[colorlinks,citecolor=blue,urlcolor=blue]{hyperref}
\RequirePackage{graphicx}

\startlocaldefs
\numberwithin{equation}{section}
\theoremstyle{plain}
\newtheorem{thm}{Theorem}[section]

\newtheorem{prop}{Proposition}[section]

\newtheorem{rem}{Remark}[section]

\newtheorem{cor}{Corollary}[section]

\newcommand{\Lower}[2]{\smash{\lower #1 \hbox{#2}}}
\newcommand{\ben}{\begin{enumerate}}
\newcommand{\een}{\end{enumerate}}
\newcommand{\bi}{\begin{itemize}}
\newcommand{\ei}{\end{itemize}}

\newcommand{\Hfun}[5]
{H_{#2}^{#1} \left[#3\left|
\begin{array}{l} #4 \vspace*{.05in}\\ #5
\end{array}\right.\right]} 

\newcommand{\PitmanDiagram}[5]
{#1 \qquad
\begin{array}{c}
#2\\
\xrightarrow{#5}\\
\xleftarrow{#5}\\
#4
\end{array}
\qquad #3
}

\endlocaldefs

\begin{document}

\begin{frontmatter}
\title{Inverse Clustering of Gibbs Partitions via independent fragmentation and dual dependent coagulation operators}
\runtitle{Fragmenting Gibbs Partitions}
\runauthor{M.-W. Ho, L.~F. James and J.~W. Lau}

\begin{aug}
\author[A]{\fnms{Man-Wai}~\snm{Ho}\ead[label=e1]{imjasonho@ust.hk}},
\author[A]{\fnms{Lancelot F.}~\snm{James}\ead[label=e2]{lancelot@ust.hk}
}
\and
\author[B]{\fnms{John W.}~\snm{Lau}\ead[label=e3]{john.w.lau@googlemail.com}}
\address[A]{{Department of Information Systems, Business Statistics
and Operations Management}, The Hong Kong University of Science and Technology\printead[presep={,\ }]{e1,e2}}

\address[B]{Department of Mathematics and Statistics, University of Western Australia\printead[presep={,\ }]{e3}}
\end{aug}

\begin{abstract}
Gibbs partitions of the integers generated by stable subordinators of index $\alpha\in(0,1)$ form remarkable classes of random partitions where in principle much is known about their properties, including practically effortless obtainment of otherwise complex asymptotic results potentially relevant to applications in  general combinatorial stochastic processes, random tree/graph growth models and Bayesian statistics. This class includes the well-known models based on the two-parameter Poisson-Dirichlet distribution which forms the bulk of explicit applications. This work continues efforts to provide interpretations for a larger classes of Gibbs partitions by embedding important operations within this framework. Here we address the formidable problem of extending the dual, infinite-block, coagulation/fragmentation results of Pitman~\cite{Pit99Coag}, where in terms of coagulation they are based on independent two-parameter Poisson-Dirichlet distributions, to all such Gibbs (stable Poisson-Kingman) models. Our results create nested families of Gibbs partitions, and corresponding mass partitions, over any $0<\beta<\alpha<1.$ We primarily focus on the fragmentation operations, which remain independent in this setting, and corresponding remarkable calculations for Gibbs partitions derived from that operation. We also present definitive results for the dual coagulation operations, now based on our construction of dependent processes, and demonstrate its relatively simple application in terms of Mittag-Leffler and generalized gamma models. The latter demonstrates another approach to recover the duality results in~\cite{Pit99Coag}. 
\end{abstract}



\end{frontmatter}

\section{Introduction}
Gibbs (random) partitions, as developed in~\cite[Theorem 8]{Pit03} and subsequently~\cite{Gnedin06,Pit06}, say, $\{A_{1},\ldots,A_{K_{n}}\}$ of the integers $[n]=\{1,2,\ldots,n\},$ where $K_{n}\leq n,$ form remarkable classes of infinitely exchangeable random partitions~\cite[Section~2.2]{Pit06}, 
whose distribution is defined consistently for each $n\in \mathbb{N}=\{1,2,\ldots\}$ as the restriction of an exchangeable partition probability function~(EPPF) on $\mathbb{N}$ that has the distinguishing feature of having Gibbs (product) form for each $n.$ In principle, the Gibbs form is attractive in terms of practical implementation, and much is known about the properties of this class, including practically effortless obtainment of otherwise complex asymptotic results potentially relevant to applications in combinatorial stochastic processes, random tree/graph growth models and Bayesian statistics. See, for example,~\cite{Bacallado,BerFrag,CaronFox,Deblasi,Feng,GoldschmidtHaas2015,
Haas,Haas2,HJL2,SvanteUrn,SvanteKuba,Pek2017jointpref,PYaku,
RembartWinkel2016a,Wood}. 

The class is defined for each $\alpha\in(-\infty,1),$ here we shall focus on the case where $\alpha\in(0,1),$ which may be generated by an $\alpha$-stable subordinator in relation to the work~\cite{PPY92,PY97,Pit03,Pit06} on excursion lengths of Bessel processes. In particular, let $\mathcal{B}:=(B_{t}, t>0)$ denote a strong Markov process on
$\mathbb{R}$ whose normalized ranked lengths of excursions,
$(P_{i})\in \mathcal{P}_{\infty}=\{\mathbf{s}=(s_{1},s_{2},\ldots):s_{1}\ge
s_{2}\ge\cdots\ge 0 {\mbox { and }} \sum_{k=1}^{\infty}s_{k}=1\},
$ follow a Poisson Dirichlet law with parameters $(\alpha,0)$, for
$0<\alpha<1,$ as discussed in Pitman and Yor~\cite{PY97}. Denote
this law as $\mathrm{PD}(\alpha,0)$ on the space of mass partitions summing to one, $\mathcal{P}_{\infty}.$ Let $(L_{t}; t>0)$ denote
its local time starting at $0,$ and let $T_{\alpha}(\ell)=\inf\{t:L_{t}>\ell\}, \ell\ge 0$, denote its inverse local time. In this
case, $\mathbf{T}_{\alpha}:=(T_{\alpha}(t),t\ge 0)$ is an $\alpha$-stable subordinator. For each $s,$ 
$T_{\alpha}(s)\overset{d}=s^{
{1}/{\alpha}}T_{\alpha}(1),$ where $T_{\alpha}(1):=T_{\alpha}$ is the inverse local time at $1$ with density $f_{\alpha}(t)$ and Laplace transform $\mathbb{E}[{\mbox e}^{-\lambda T_{\alpha}}]={\mbox e}^{-\lambda^{\alpha}}.$ Due to the scaling identity
(see~\cite{PY92}),
\begin{equation}
L_{1}\overset{d}=\frac{L_{t}}{t^{\alpha}}\overset{d}=\frac{s}{[{{T_{\alpha}(s)]}^{\alpha}}}\overset{d}=T^{-\alpha}_{\alpha},
\label{scaling}
\end{equation}
the local time up to time $1,$ $L_{1}\overset{d}=T^{-\alpha}_{\alpha},$ follows a Mittag-Leffler distribution with density $g_{\alpha}(z):=f_{\alpha}(z^{-
1/\alpha})z^{-
{1}/\alpha-1}/\alpha,$ satisfying
\begin{equation}
L_{1}:=\Gamma(1-\alpha)^{-1}\lim_{\epsilon\rightarrow
0}\epsilon^{\alpha}|\{i:P_{i}\ge \epsilon\}|~\mathrm{a.s.}
\label{inverselocaltime}
\end{equation}
The process $F_{\alpha,0}(y):=\sum_{k=1}^{\infty}P_{k}\mathbb{I}_{\{U_{k}\leq y\}},$
for $(U_{k})\overset{iid}\sim\mathrm{Uniform}[0,1]$ independent of $(P_{i})\sim \mathrm{PD}(\alpha,0)$, is called a $\mathrm{PD}(\alpha,0)$ exchangeable bridge, see \cite{BerFrag} for this terminology, with the equivalence
$$
F_{\alpha,0}(y)\overset{d}=\frac{T_{\alpha}(y)}{T_{\alpha}(1)}.
$$
In turn, sampling $X_{1},\ldots,X_{n}|F_{\alpha,0}\overset{iid}\sim F_{\alpha,0},$
that is, $(X_{i}\overset{d}=F_{\alpha,0}^{-1}(U'_{i});$ $U'_{i}\overset{iid}\sim\mathrm{Uniform}[0,1], i\in [n])$, produces $K_{n}=k$ unique values $(Y_{1},\ldots,Y_{k})\overset{iid}\sim\mathrm{Uniform}[0,1]$ and a $\mathrm{PD}(\alpha,0)$-partition of $[n],$ $\{A_{1},\ldots,A_{k}\},$ where $A_{j}=\{i:X_{i}=Y_{j}\},$
with size $|A_{j}|=n_{j},$ for $j=1,\ldots,k,$ with an EPPF denoted by the $\mathrm{PD}(\alpha,0)-\mathrm{EPPF}$,
\begin{equation}
p_{\alpha}(n_{1},\ldots,n_{k})=\frac{\alpha^{k-1}\Gamma(k)}{\Gamma(n)}\prod_{j=1}^{k}(1-\alpha)_{n_{j}-1},
\label{canonEPPF}
\end{equation}
where, for any non-negative integer $x$, $(x)_n = x(x+1)\cdots(x+n-1) = {\Gamma(x+n)}/{\Gamma(x)}$ denotes the Pochhammer symbol. From this, the probability of the number of blocks $K_{n}=k$ can be expressed as $\mathbb{P}_{\alpha,0}^{(n)}(k):=\mathbb{P}_{\alpha,0}(K_n = k)=
{\alpha^{k-1}\Gamma(k)} S_\alpha(n,k) / {\Gamma(n)}$, where $S_\alpha(n,k) =  
[{\alpha^k k!}]^{-1} \sum_{j=1}^k (-1)^j \binom{k}{j} (-j\alpha)_n$ denotes the generalized Stirling number of the second kind.

\subsection{$\alpha$-stable Poisson-Kingman distributions and Gibbs partitions}
Now conditioning $(P_{i})$ on $T_{\alpha}=t$~(or $L_{1}=t^{-\alpha}$) leads to the distribution of $(P_{i})|T_{\alpha}=t\sim \mathrm{PD}(\alpha|t)$, and for $h(t)$ a non-negative function with $\mathbb{E}[h(T_{\alpha})]=1,$ one may, as in~\cite{Pit03}, define the $\alpha$-stable Poisson-Kingman distribution with mixing distribution $\nu(dt)/dt=h(t)f_{\alpha}(t),$ and write $(P_{\ell})\sim \mathrm{PK}_{\alpha}(h\cdot f_{\alpha})$ defined as 
$$
\mathrm{PK}_{\alpha}(h\cdot f_{\alpha}):=\int_{0}^{\infty}\mathrm{PD}(\alpha|t)h(t)f_{\alpha}(t)dt=\int_{0}^{\infty}\mathrm{PD}(\alpha|s^{-\frac{1}{\alpha}})h(s^{-\frac{1}{\alpha}})g_{\alpha}(s)ds.
$$
Setting $h(t)=t^{-\theta}/\mathbb{E}[T^{-\theta}_{\alpha}]$, for $\theta>-\alpha$, leads to $(P_{\ell})\sim \mathrm{PD}(\alpha,\theta),$ corresponding to the important two-parameter Poisson-Dirichlet distribution as described in~\cite{PPY92, Pit96,Pit03,Pit06,PY97}, whose size biased re-arrangement, say, $(\tilde{P}_{\ell})\sim \mathrm{GEM}(\alpha,\theta),$ where $\mathrm{GEM}(\alpha,\theta)$ is the two-parameter Griffiths-Engen-McCloskey distribution, and is widely used in applications~\cite{BerFrag,IJ2001,PPY92,Pit96,Pit06}. 
The inverse local time at $1$ of a process with lengths $(P_{\ell})\sim \mathrm{PD}(\alpha,\theta),$ say, $T_{\alpha,\theta},$ has density $f_{\alpha,\theta}(t)=t^{-\theta}f_{\alpha}(t)/\mathbb{E}[T^{-\theta}_{\alpha}]$, and the corresponding local time at $1$ or its $\alpha$-diversity, $T^{-\alpha}_{\alpha,\theta}\sim \mathrm{ML}(\alpha,\theta),$ denoting that it has a
 $(\alpha,\theta)$ generalized Mittag-Leffler distribution with density $g_{\alpha,\theta}(s)=s^{\theta/\alpha}g_{\alpha}(s)/\mathbb{E}[T^{-\theta}_{\alpha}]$.
When $(P_{\ell})\sim \mathrm{PK}_{\alpha}(h\cdot f_{\alpha})$, $F(y)=\sum_{k=1}^{\infty}P_{k}\mathbb{I}_{\{U_{k}\leq y\}}$ is its corresponding bridge defined similarly as $F_{\alpha,0}$, and sampling $n$ variables from $F$, as in the $\mathrm{PD}(\alpha,0)$ case, leads to the general class of $\alpha\in(0,1)$ Gibbs partitions of $[n]$, with an EPPF , denoted by $\mathrm{PK}_{\alpha}(h\cdot f_{\alpha})-\mathrm{EPPF},$ as in ~\cite{Gnedin06,Pit03},
\begin{equation}
p^{[\nu]}_{\alpha}(n_{1},\ldots,n_{k})={\Psi}^{[\alpha]}_{n,k}\times p_{\alpha}(n_{1},\ldots,n_{k}),
\label{VEPPF}
\end{equation}
where, using the interpretation of the expressions derived by \cite{Gnedin06,Pit03} in~\cite{HJL2},
$$
{\Psi}^{[\alpha]}_{n,k} =\mathbb{E}[h(T_{\alpha})|K_{n}=k]=\mathbb{E}\bigg[h\big(Y^{(n-k\alpha)}_{\alpha,k\alpha}\big)\bigg].
$$ 
In the first expectation, $T_{\alpha}|K_{n}=k$ is evaluated for the $\mathrm{PD}(\alpha,0)$ case with $K_{n}\sim 
\mathbb{P}_{\alpha,0}^{(n)}(k).$ The second equality follows from the fact that such a conditional random variable equates in distribution 
to a variable 
$Y^{(n-k\alpha)}_{\alpha,k\alpha},$ with density $f^{(n-k\alpha)}_{\alpha,k\alpha}(t),$ such that~(pointwise), as in~\cite[eq. (2.13), p. 323]{HJL2},
\begin{equation}
Y^{(n-k\alpha)}_{\alpha,k \alpha}\overset{d}=\frac{T_{\alpha,k\alpha}}{B_{k\alpha,n-k\alpha}}=\frac{T_{\alpha,n}}{B^{\frac1\alpha}_{k,\frac{n}{\alpha}-k}},
\label{jamesidspecial}
\end{equation}
where variables in each ratio are independent, and throughout, $B_{a,b}$ denotes a $\mathrm{Beta}(a,b)$ random variable. It is noteworthy that~(\ref{jamesidspecial}) indicates that Mittag-Leffler variables play a role in the general $\alpha$ class of Gibbs partitions.

\subsection{Interpreting Gibbs partitions via infinite block fragmentation and coagulation operations}
While the general class of $\mathrm{PK}_{\alpha}(h\cdot f_{\alpha})$ 
distributions and the corresponding Gibbs partitions exhibit many desirable properties, most choices of $h(t)$ do not have any particular interpretation. The most notable exceptions are the important $\mathrm{PD}(\alpha,\theta)$ distributions which dominate the broad literature. There are some additional examples, such as the generalized gamma and Mittag-Leffler classes. Of particular interest to us are classes that correspond to nested families of mass partitions, whose marginal distributions follow some explicit collection of $\mathrm{PK}_{\alpha}$ distributions, and whose Markovian dynamics~(dependence structure) may be described by some operations on $\mathcal{P}_{\infty},$ and hence equivalently on spaces of integer partitions. The two most notable examples of such families are the nested families of the form $((P^{(j)}_{\ell}),j=0,1,2\ldots)$ represented by the Poisson-Dirichlet laws $(\mathrm{PD}(\alpha,\theta+j\alpha),j=0,1,\ldots),$ and the laws $(\mathrm{PD}(\alpha,\theta+j),j=0,1,\ldots).$ The first such collection is related to (dual) operations of size-biased deletion and insertion, as described in~\cite[Propositions 34 and 35]{PY97}, and the latter may be constructed by a single block $\mathrm{PD}(\alpha,1-\alpha)$ fragmentation operation, fragmenting successively the size-biased pick of the indicated families, leading to, in the initial case of $\mathrm{PD}(\alpha,\theta),$ an increase of $\theta$ to $\theta+1,$ and inversely by a dual coagulation operation as described in~\cite{DGM}. It is notable that the operations of size-biased deletion, as described in~\cite{PPY92,Pit03,Pit06}, and single-block fragmentation can be applied in principle to general families taking values on $\mathcal{P}_{\infty},$ whereas the latter (dual) operations involve the usage of independent beta distributed variables, which are particular to the $\mathrm{PD}(\alpha,\theta)$ distribution. Common descriptions, for extensions of size-biased deletion involving general
$\mathrm{PK}_{\alpha}(h\cdot f_{\alpha}),$ can be deduced from~\cite{PPY92, Pit03}, whereas fragmentation by $\mathrm{PD}(\alpha,1-\alpha)$ has been treated in~\cite{HJL2,JamesFragblock}. Results for specific examples of $h(t),$ requiring more detailed analysis, are also discussed in those and related works.

This work continues efforts to provide interpretations for larger classes of Gibbs partitions by embedding important operations within this framework. Here we address the formidable problem of extending the
dual, infinite block, coagulation/fragmentation results of Pitman~\cite{Pit99Coag}, that is, for $0<\beta<\alpha<1,$ 
a dual relationship between $V_{\beta,\theta}\sim \mathrm{PD}(\beta,\theta)$ and $V_{\alpha,\theta}\sim \mathrm{PD}(\alpha,\theta),$ with appropriate coagulation and fragmentation operations indicated, as in~\cite[Section 5.5]{Pit06}, by the following diagram
\begin{equation}
\PitmanDiagram{\mathrm{PD}(\alpha,\theta)}
{\mathrm{PD}(\beta/\alpha,\theta/\alpha)-\mathrm{COAG}}
{\mathrm{PD}(\beta,\theta)}
{\mathrm{PD}(\alpha,-\beta)-\mathrm{FRAG}}
{\hspace*{1in}}
\label{pitdual}
\end{equation}
This creates nested families of mass/integer partitions having laws $(\mathrm{PD}(\alpha,\theta),0<\alpha<1).$ Setting $\theta=0,$ $\alpha={\mbox e}^{-t}$ and $\beta={\mbox e}^{-(t+s)},$ the (homogeneous) continuous-time Markov operator $(\mathrm{PD}(e^{-s},0)-\mathrm{COAG}, s\ge 0)$ coincides with the semi-group of the Bolthausen-Sznitman coalescent~\cite{Bolt}, as discussed in~\cite{BerFrag,Pit99Coag,Pit06}. Briefly, following~\cite{BerFrag,Pit06}, we describe the $\mathrm{PD}(\alpha,-\beta)-\mathrm{FRAG}$ fragmentation operator, hereafter denoted as $\mathrm{FRAG}_{\alpha,-\beta},$ and corresponding $\mathrm{PD}(\beta/\alpha,\theta/\alpha)-\mathrm{COAG}$ coagulation operator. For $(P_{i})\in \mathcal{P}_{\infty},$ the $\mathrm{FRAG}_{\alpha,-\beta}$ operator is defined as, for $\mathrm{Rank}$ denoting the ranked re-arrangement of masses, 
\begin{equation}
\mathrm{FRAG}_{\alpha,-\beta}((P_{i})):=\mathrm{Rank}(P_{i}(\hat{Q}^{(i)}_{\ell}),i\ge 1)\in \mathcal{P}_{\infty},
\label{fragop}
\end{equation}
where, for each $i,$ $\hat{Q}^{(i)}:=(\hat{Q}^{(i)}_{\ell})\in\mathcal{P}_{\infty},$ the collection $(\hat{Q}^{(i)},i\ge 1)$ are iid $\mathrm{PD}(\alpha,-\beta)$ mass partitions, taken independent of the input $(P_{i}).$ There is also the property $\mathrm{FRAG}_{\alpha,-\sigma}\circ \mathrm{FRAG}_{\sigma,-\beta}=\mathrm{FRAG}_{\alpha,-\beta}$
for all $0<\beta<\sigma<\alpha<1.$ In terms of partitions $\{A_{1},\ldots,A_{k}\}$, one is shattering each block of the partition, say, $A_{j}$ with size $|A_{j}|,$ by an independent $\mathrm{PD}(\alpha,-\beta)$ partition of $|A_{j}|$ elements.
We now proceed to first describe the coagulation for more general independent laws on $\mathcal{P}_{\infty},$ by using the equivalence in terms of compositions of exchangeable bridges as in~\cite[Lemma 5.18]{Pit06} or \cite{BerFrag}. Suppose that $F_{V}$ and $G_{Q}$ are independent exchangeable bridges defined for $V$ and $Q$ in $\mathcal{P}_{\infty},$ with respective distributions $\mathbb{P}_{V}$ and $\mathbb{P}_{Q}.$ Then, for the input $\tilde{V},$  $V:=\mathbb{P}_{Q}-\mathrm{COAG}(\tilde{V})\in \mathcal{P}_{\infty}$ is equivalent to the ranked re-arrangement of masses formed by the composition $F_{V}(y)=F_{\tilde{V}}(G_{Q}(y)),$ for each $y\in[0,1],$ or simply $F_{V}=F_{\tilde{V}}\circ G_{Q},$ where $V$ has distribution $\mathbb{P}_{V}.$ Letting $F_{\alpha,\theta}$ denote a $\mathrm{PD}(\alpha,\theta)$ bridge, the $\mathrm{PD}(
{\beta}/{\alpha},
{\theta}/{\alpha})-\mathrm{COAG}$ result in~\cite{Pit99Coag}, as indicated from left to right in~(\ref{pitdual}), corresponds to $F_{\beta,\theta}=F_{\alpha,\theta}\circ G_{
{\beta}/{\alpha},
{\theta}/{\alpha}}.$ 
\begin{rem}
Note $F_{\beta}:=F_{\beta,0}$ and the same is true for other variables with $\theta=0.$
\end{rem}
In this work, for $0<\beta<\alpha<1$ and general  $V\sim \mathrm{PK}_{\beta}(h\cdot f_{\beta}),$ where $\mathbb{E}[h(T_{\beta})]=1,$ we shall apply the same independent $\mathrm{FRAG}_{\alpha,-\beta}$ operator leading to explicit identification of laws, and calculations for nested families of mass partitions and corresponding Gibbs partitions over $\alpha\in(0,1).$ This represents the primary focus of our work, however we also describe in detail how to construct a natural dual coagulation operation via dependent compositions of bridges and corresponding mass partitions, and demonstrate how easy it can be applied. We note that, in the literature, coagulation operations of this sort are generally defined for independent processes as in~\cite[Lemma 5.18]{Pit06}. While, in general, fragmentation and coagulation operations are clearly defined and hence straightforward to apply, the formidable challenge is to identify the resulting relevant distributions, and find ones with tractability. The duality result of \cite{Pit99Coag}, as described in~(\ref{pitdual}), is achieved by working with the EPPF's of the corresponding random partitions of $[n].$ In principle one may try such an approach to identify the various laws associated with $V\sim \mathrm{PK}_{\beta}(h\cdot f_{\alpha}),$ or attempt via the explicit constructions of exchangeable bridges. However, neither approach seems feasible. Here working on the space of mass partitions, we show, in Sections~\ref{sec:2frag} and~\ref{sec:3coag}, that the classes of (marginal) distributions $\tilde{V}\sim\mathrm{PK}_{\alpha}(\tilde{h}_{
{\beta}/{\alpha}}\cdot f_{\alpha})$ and $Q\sim\mathrm{PK}_{\beta/\alpha}(\hat{h}_{\alpha}\cdot f_{
{\beta}/{\alpha}}),$ with 
\begin{equation}
\tilde{h}_{\frac{\beta}{\alpha}}(v):=\mathbb{E}_{\frac{\beta}{\alpha}}\left[h\bigg(vT^{\frac{1}{\alpha}}_{\frac{\beta}{\alpha}}\bigg)\right]\qquad
{\mbox { and }}\qquad \hat{h}_{\alpha}(y):=\mathbb{E}_{\alpha}\left[h\big(T_{\alpha}y^{\frac{1}{\alpha}}\big)\right],
\end{equation}
may be interpreted as being equivalent in distribution to those arising from $\tilde{V}=\mathrm{FRAG}_{\alpha,-\beta}(V),$ and the marginal distribution of the corresponding coagulator which leads to $F_{V}=F_{\tilde{V}}\circ G_{Q},$ respectively. In Sections~\ref{sec:41GibbsPartFrag}-\ref{sec:43sampling}
, we obtain remarkable calculations for Gibbs partitions, and related identities, derived from the $\mathrm{FRAG}_{\alpha,-\beta}$ operation. Sections~\ref{sec:44PKhtilde} and~\ref{sec:45brownfrag}, provide interesting  identities, fixed point equations, and asymptotic results, which has interpretations in relation to $D-$ and $T-\mathrm{partition}$ intervals as discussed in~\cite{AldousT,Pit06}. 
Sections~\ref{ss:46Mittagsection} and~\ref{sec:5dualMittagsection} provide developments in relation to the Mittag-Leffler class~\cite{HJL2}. Section~\ref{sec:6CoagFragGGmodels} demonstrates how to use our dual coagulation and fragmentation operations to easily identify all the relevant laws, and constructs duality results for generalized gamma models, and size biased extensions. This presents another approach to recover the duality results in~\cite{Pit99Coag} for all $\theta>-\beta.$ 
 
For nested models primarily related to the fragmentation operator see~\cite{Basdevant,CurienHaas}. In addition,~\cite{Wood}, see also~\cite{Gasthaus}, applies the coag/frag duality on the space of partitions of $[n]$ to $\infty$-gram natural language models. This represents an application in Bayesian statistical machine learning involving the usage of inverse clustering (via $\mathrm{FRAG}_{\alpha,-\beta}$ fragmentation) and merging (via $\mathrm{PD}(\alpha,\theta)$ coagulation) on the space of partitions of $[n].$ Related to this,~\cite{JamesHnetwork} constructs (nested) hierarchical network/graph models using the coagulation fragmentation operations in~\cite{Pit99Coag} and also~\cite{DGM}.
For some other references on Gibbs-partitions and $\alpha$-stable Poisson-Kingman models, see~\cite{Bacallado,CaronFox,Deblasi,GriffithsSpano,HJL, HJL2,LomeliFavaro,PYaku}. See~\cite{Haas,Haas2} for other occurrences of the coag/frag operators in the $\mathrm{PD}(\alpha,\theta)$ setting.

\section{$\mathrm{PD}_{\alpha,-\beta}$ fragmentation of $\mathrm{PK}_{\beta}(h\cdot f_{\beta})$ mass partitions}\label{sec:2frag}
In order to achieve our results, we work with independent stable subordinators $\mathbf{T}_{\alpha}$ and $\mathbf{T}_{
{\beta}/{\alpha}},$ and representations of their relevant quantities under the independent $\mathrm{PD}(\alpha,0)$ and $\mathrm{PD}(
{\beta}/{\alpha},0)$ distributions. The corresponding independent local time processes are $(L_{\alpha}(t),t>0)$ and $(L_{
{\beta}/{\alpha}}(t),t\ge 0)$, satisfying~(\ref{scaling}) and~(\ref{inverselocaltime}), with local times at $1,$ denoted respectively as $L_{1,\alpha}\overset{d}=T^{-\alpha}_{\alpha}$ and $L_{1,
{\beta}/{\alpha}}\overset{d}=T^{-
{\beta}/{\alpha}}_{
{\beta}/{\alpha}}$, playing the role of $L_{1}$, as we have described, and otherwise following the more detailed description in~\cite{BerPit2000}, as it relates to the special case of the duality~\cite[Theorem 14 and Corollary 15]{Pit99Coag}.
In particular, $V_{\beta}\sim \mathrm{PD}(\beta,0)$ is formed 
by the independent coagulation,
$\mathrm{V}_{\beta}=\mathrm{PD}(\frac{\beta}{\alpha},0)-\mathrm{COAG}((V_{\alpha}))$, equivalent to, as in~\cite{BerFrag,BerLegall00,BerPit2000,Pit06},
\begin{equation}
\left(F_{\beta}(u)=F_{\alpha}\big(F_{\frac{\beta}{\alpha}}(u)\big), u\in[0,1]\right)\overset{d}=\left(\frac{T_{\alpha}\big(T_{\frac{\beta}{\alpha}}(u)\big)}{T_{\alpha}\big(T_{\frac{\beta}{\alpha}}(1)\big)},u\in[0,1]\right)
\label{subcoag} 
\end{equation}
and has local time at $1,$ $L_{1,\beta}=L_{
{\beta}/{\alpha}}(L_{1,\alpha})$ and inverse local time at $1,$ 
$T_{\beta}:=T_{\beta}(1)=T_{\alpha}(T_{
{\beta}/{\alpha}}(1)).$ Conversely, $V_{\alpha}=\mathrm{FRAG}_{\alpha,-\beta}(V_{\beta})\sim \mathrm{PD}(\alpha,0).$

\begin{rem}Note there is the well known distributional equivalence $T_{\beta}\overset{d}=T_{\alpha}\times T^{
{1}/{\alpha}}_{
{\beta}/{\alpha}}\overset{d}=T_{
{\beta}/{\alpha}}\times T^{
{\alpha}/{\beta}}_{\alpha}.$ However, in the case of interpretation of the $\mathrm{PD}(
{\beta}/{\alpha},0)$ coagulation, as in~(\ref{subcoag}), the order matters and thus we will only use $T_{\beta}\overset{d}=T_{\alpha}\times T^{
{1}/{\alpha}}_{
{\beta}/{\alpha}}$.
\end{rem}

Define
\begin{equation}
\omega^{(y)}_{\frac{\beta}{\alpha},\beta}(s)=\frac{\mathbb{P}
\left({\big[L_{\frac{\beta}{\alpha}}(L_{1,\alpha})\big]}^{-\frac{1}{\beta}}\in dy\left|L^{-\frac{1}{\alpha}}_{1,\alpha}=s\right.\right)}
{\mathbb{P}\big(L^{-\frac{1}{\beta}}_{1,\beta}\in dy\big)}
=\frac{\alpha y^{\alpha-1}f_{\frac{\beta}{\alpha}}\big({(y/s)}^{\alpha}\big)}{s^{\alpha}f_{\beta}(y)}
\label{ratioden}
\end{equation}
such that the conditional distribution of $L_{1,\alpha}|L_{1,\beta}$ may be expressed in terms of that of the transformed variable $L^{-
{1}/{\alpha}}_{1,\alpha}|L^{-
{1}/{\beta}}_{1,\beta}$  as,
\begin{equation}
\mathbb{P}
\left(L^{-\frac{1}{\alpha}}_{1,\alpha}\in ds\left|\big[L_{\frac{\beta}{\alpha}}(L_{1,\alpha})\big]^{-\frac{1}{\beta}}=y\right.\right)/ds=\omega^{(y)}_{\frac{\beta}{\alpha},\beta}(s)f_{\alpha}(s),
\label{mixdenfrag}
\end{equation}
which is equivalent to the conditional density of $T_{\alpha}$ given $T_{\alpha}\times T^{
{1}/{\alpha}}_{
{\beta}/{\alpha}}=y$.

\begin{thm}\label{fragtheorem}
Let $V\sim \mathrm{PK}_{\beta}(h\cdot f_{\beta})$ with local time at $1,$ say $L_{1,V},$ having density $h(\ell^{-
{1}/{\beta}})g_{\beta}(\ell).$ For any choice of $0<\beta<\alpha<1,$ let $\mathrm{FRAG}_{\alpha,-\beta}(\cdot)$ denote an $\mathrm{PD}(\alpha,-\beta)$ fragmentation operator independent of $V.$ Then,
\begin{enumerate}
\item[(i)] $\tilde{V}=\mathrm{FRAG}_{\alpha,-\beta}(V)\sim \mathrm{PK}_{\alpha}\bigg(\tilde{h}_{\frac{\beta}{\alpha}}\cdot f_{\alpha}\bigg)$ where
$$
\tilde{h}_{\frac{\beta}{\alpha}}(v):=\mathbb{E}_{\frac{\beta}{\alpha}}\left[h\bigg(vT^{\frac{1}{\alpha}}_{\frac{\beta}{\alpha}}\bigg)\right].
$$
That is, it has a local time at time $1,$ say, 
$\tilde{L}_{1,\tilde{V}}$, with density $\tilde{h}_{\frac{\beta}{\alpha}}(s^{-\frac{1}{\alpha}})g_{\alpha}(s)$.
\item[(ii)]
$\tilde{V}|L_{1,V}=y^{-\beta}$ has the $\alpha$-stable Poisson-Kingman distribution with, for each $y,$ and any $h(y),$ mixing density
$$
\mathbb{P}
\left(\tilde{L}^{-\frac{1}{\alpha}}_{1,\tilde{V}}\in ds\left|L^{-\frac{1}{\beta}}_{1,V}=y\right.\right)/ds
=\mathbb{P}
\left(L^{-\frac{1}{\alpha}}_{1,\alpha}\in ds\left|\big[L_{\frac{\beta}{\alpha}}(L_{1,\alpha})\big]^{-\frac{1}{\beta}}=y\right.\right)/ds,
$$
as in~(\ref{mixdenfrag}).
\item[(iii)]The conditional distribution of $\tilde{V}|L^{-
{1}/{\beta}}_{1,V}=y$ is equivalent to the distribution of $V_{\alpha}|L^{-
{1}/{\beta}}_{1,\beta}=y,$ which is 
\begin{equation}
\mathrm{PD}_{\alpha|\beta}(\alpha|y):=\int_{0}^{\infty}\mathrm{PD}(\alpha|s)\omega^{(y)}_{\frac{\beta}{\alpha},\beta}(s)f_{\alpha}(s)ds
=\mathrm{PK}_{\alpha}\left(\omega^{(y)}_{\frac{\beta}{\alpha},\beta}\cdot f_{\alpha}\right),
\label{condlocaltime}
\end{equation}
for $\omega^{(y)}_{\frac{\beta}{\alpha},\beta}(s)$ defined in~(\ref{ratioden}).
\end{enumerate}
\end{thm}

\begin{proof}
Let 
$\mathbb{\hat{E}}^{(\beta,0)}_{(\alpha,-\beta)}$ denote the expectation with respect to the joint law of $\big(V,(\hat{Q}^{(k)}, k\ge 1)\big)$
where $V\overset{d}=V_{\beta}\sim \mathrm{PD}(\beta,0)$ with local time at $1,$ $L_{1,\beta}$, with density $g_{\beta}(\ell)$, and independent of this, 
$\big(\hat{Q}^{(k)}, k\ge 1\big)$ are iid $\mathrm{PD}(\alpha,-\beta)$ mass partitions. Consider $V\sim \mathrm{PK}_{\alpha}(h\cdot f_{\beta}).$ The distribution of $\tilde{V}=\mathrm{FRAG}_{\alpha,-\beta}(V)$ is characterized, for a measurable function $\Omega,$ by
\begin{equation}
\mathbb{E}
\left[\Omega\big(\mathrm{FRAG}_{\alpha,-\beta}(V)\big)\right]
=\mathbb{\hat{E}}^{(\beta,0)}_{(\alpha,-\beta)}
\left[\Omega\big(\mathrm{FRAG}_{\alpha,-\beta}(V_{\beta})\big)h\big(L^{-\frac{1}{\beta}}_{1,\beta}\big)\right].
\label{character1}
\end{equation}
But, from \cite{BerPit2000,Pit99Coag}, as described in~(\ref{pitdual}), this is equivalent to, for $V_{\alpha}\sim \mathrm{PD}(\alpha,0),$
\begin{equation}
\mathbb{E}
\left[\Omega(V_{\alpha})h\bigg({\big[L_{\frac{\beta}{\alpha}}(L_{1,\alpha})\big]}^{-\frac{1}{\beta}}\bigg)\right].
\label{character2}
\end{equation}
It follows that $V_{\alpha}|L^{-
{1}/{\alpha}}_{1,\alpha}=s,{[L_{
{\beta}/{\alpha}}(L_{1,\alpha})]}^{-
{1}/{\beta}}=y$ has distribution $\mathrm{PD}(\alpha|s),$ not depending on $y.$ Using this, the scaling property $L_{
{\beta}/{\alpha}}(s^{-\alpha})\overset{d}=L_{1,
{\beta}/{\alpha}}\times s^{-\beta},$ and elementary arguments to describe the joint density, it follows that the expectation~(\ref{character2}) can be expressed as 
$$
\int_{0}^{\infty}\left[\int_{0}^{\infty}
\mathbb{E}[\Omega(V_{\alpha})|T_{\alpha}=s]\omega^{(y)}_{\frac{\beta}{\alpha},\beta}(s)f_{\alpha}(s)ds\right]h(y)f_{\beta}(y)
dy,
$$
which can also be expressed as,
$$
\int_{0}^{\infty}
\mathbb{E}[\Omega(V_{\alpha})|T_{\alpha}=s]\tilde{h}_{\frac{\beta}{\alpha}}(s)f_{\alpha}(s)ds
$$
for $\tilde{h}_{\frac{\beta}{\alpha}}(s):=\mathbb{E}_{\frac{\beta}{\alpha}}\left[h\bigg(sT^{\frac{1}{\alpha}}_{\frac{\beta}{\alpha}}\bigg)\right]
=\int_{0}^{\infty}\omega^{(y)}_{\frac{\beta}{\alpha},\beta}(s)h(y)f_{\beta}(y)dy,$ yielding the results. 
\end{proof}

Recall that for any $0<\beta<\sigma<\alpha<1,$ $\mathrm{FRAG}_{\alpha,-\beta}=\mathrm{FRAG}_{\alpha,-\sigma}\circ \mathrm{FRAG}_{\sigma,-\beta}.$

\begin{cor}
Suppose that $0<\beta<\sigma<\alpha<1,$  $V\sim \mathrm{PK}_{\beta}(h\cdot f_{\beta})$ with $\mathbb{E}[h(T_{\beta})]=1.$ Then, $V'=\mathrm{FRAG}_{\sigma,-\beta}(V)\sim 
\mathrm{PK}_{\sigma}\bigg(\tilde{h}_{\frac{\beta}{\sigma}}\cdot f_{\sigma}\bigg)$ where
$
\tilde{h}_{\frac{\beta}{\sigma}}(s):=\mathbb{E}_{\frac{\beta}{\sigma}}\left[h\bigg(sT^{\frac{1}{\sigma}}_{\frac{\beta}{\sigma}}\bigg)\right].
$ 
Hence, $\tilde{V}=\mathrm{FRAG}_{\alpha,-\sigma}(V')\sim \mathrm{PK}_{\alpha}\bigg(\tilde{h}_{\frac{\beta}{\alpha}}\cdot f_{\alpha}\bigg)$ where
$$
\tilde{h}_{\frac{\beta}{\alpha}}(v):=
\mathbb{E}_{\frac{\sigma}{\alpha}}\left[\tilde{h}_{\frac{\beta}{\sigma}}\bigg(vT^{\frac{1}{\alpha}}_{\frac{\sigma}{\alpha}}\bigg)\right]
=\mathbb{E}_{\frac{\beta}{\alpha}}\left[h\bigg(vT^{\frac{1}{\alpha}}_{\frac{\beta}{\alpha}}\bigg)\right].
$$
\end{cor}

\section{Duality via dependent coagulation}\label{sec:3coag}
We now describe how to construct dependent coagulations to complete the dual process of recovering $V\sim \mathrm{PK}_{\beta}(h\cdot f_{\beta})$ from the coagulation of $\tilde{V}=\mathrm{FRAG}_{\alpha,-\beta}(V)\sim \mathrm{PK}_{\alpha}(\tilde{h}_{\frac{\beta}{\alpha}}\cdot f_{\alpha}).$
Our results show how specification of $h(t)$ leads to a prescription to identify the laws of $V,\tilde{V},Q$ without guess-work. 

Recall that for the independent mass partitions $(V_{\alpha},Q_{\frac{\beta}{\alpha}})$  described in~(\ref{subcoag}), the process of coagulation yields an inverse local time at $1$ for $V_{\beta}$ to be $T_{\beta}(1)=T_{\alpha}(T_{
{\beta}/{\alpha}}(1))\overset{d}=T_{\alpha}\times T^{
{1}/{\alpha}}_{
{\beta}/{\alpha}}.$ For $\tilde{V}$ as described above, we consider the dependent pair $(\tilde{V},Q)$ with joint law, say, $\mathrm{P}^{\frac{\beta}{\alpha}}_{\alpha}(h),$ characterized by 
\begin{equation}
\mathbb{E}[\Omega(\tilde{V},Q)]=\mathbb{E}^{(\frac{\beta}{\alpha},0)}_{(\alpha,0)}\left[\Omega\big(V_{\alpha},Q_{\frac{\beta}{\alpha}}\big)h\bigg(T_{\alpha}\big(T_{\frac{\beta}{\alpha}}(1)\big)\bigg)\right],
\label{jointmeasure1}
\end{equation}
with $\mathbb{E}\big[h(T_{\beta}(1))\big]=\mathbb{E}\left[h\bigg(T_{\alpha}\big(T_{\frac{\beta}{\alpha}}(1)\big)\bigg)\right]=1$, and the notation
$\mathbb{E}^{(\frac{\beta}{\alpha},0)}_{(\alpha,0)}$ referring to an expectation evaluated under the joint law of the independent $\mathrm{PD}(\alpha,0)$ and $\mathrm{PD}(\beta/\alpha,0)$ distributions. We use this for clarity, but will suppress it when it is clear we are referring to such variables. Equivalently, by conditioning and scaling properties, the joint law of $(\tilde{V},Q)
$ is given by 
\begin{equation}
\mathrm{P}^{\frac{\beta}{\alpha}}_{\alpha}(h):=\int_{0}^{\infty}\int_{0}^{\infty}\mathrm{PD}(\alpha|s)\mathrm{PD}({\beta}/
{\alpha}|y)h\big(sy^{\frac{1}{\alpha}}\big)f_{\frac{\beta}{\alpha}}(y)f_{\alpha}(s)dyds.
\label{jointmeasure2}
\end{equation}
\begin{rem}
For further clarity, we may use the notation $\mathrm{P}^{\frac{\beta}{\alpha}}_{\alpha}(h)=\mathrm{P}^{\frac{\beta}{\alpha}}_{\alpha}\left(h\cdot(f_{\alpha}
,f_{\frac{\beta}{\alpha}})\right)$
\end{rem}

In addition, for collections of iid $\mathrm{Uniform}[0,1]$ variables $((\tilde{U}_{k}),(U_{\ell}))$ independent of $(\tilde{V},Q)$ define random distribution functions (exchangeable bridges), for $y\in[0,1]$,
\begin{equation}
F_{\tilde{V}}(y)=\sum_{k=1}^{\infty}\tilde{V}_{k}\mathbb{I}_{\left\{\tilde{U}_{k}\leq y\right\}}\qquad  {\mbox { and }} \qquad G_{Q}(y)=\sum_{\ell=1}^{\infty}Q_{\ell}\mathbb{I}_{\left\{U_{\ell}\leq y\right\}}.
\label{bridges}
\end{equation}

\begin{rem}It follows that when $h(t)=t^{-\theta}/\mathbb{E}\big[T^{-\theta}_{\beta}\big]$ for $\theta>-\beta,$ $\tilde{V}\sim \mathrm{PD}(\alpha,\theta)$ is independent of $Q\sim \mathrm{PD}(
{\beta}/{\alpha},
{\theta}/{\alpha}).$  
Hence, $F_{\tilde{V}}\overset{d}=F_{\alpha,\theta}$ and $G_{Q}\overset{d}=G_{
{\beta}/{\alpha},
{\theta}/{\alpha}}$.
\end{rem}

\begin{prop}\label{propcoag}For $0<\beta<\alpha<1,$ let $(\tilde{V},Q)$ have a joint distribution, $\mathrm{P}^{\frac{\beta}{\alpha}}_{\alpha}(h),$ specified by~(\ref{jointmeasure1}) or equivalently~(\ref{jointmeasure2}), such that $\tilde{V}\sim \mathrm{PK}_{\alpha}\bigg(\tilde{h}_{\frac{\beta}{\alpha}}\cdot f_{\alpha}\bigg),$ and 
$(F_{\tilde{V}}, G_{Q})$ are bridges defined in~(\ref{bridges}). Let $V\in \mathcal{P}_{\infty}$ be the ranked masses of the bridge defined by the composition $F_{V}:=F_{\tilde{V}}\circ G_{Q}.$ 
Then, $V$ is equivalent to the coagulation of $\tilde{V}$ by $Q$ and there are the following properties. 
\begin{enumerate}
\item[(i)]$V\sim \mathrm{PK}_{\beta}(h\cdot f_{\beta}).$
\item[(ii)]The marginal distribution of $Q~\sim \mathrm{PK}_{\beta/\alpha}(\hat{h}_{\alpha}\cdot f_{\frac{\beta}{\alpha}}),$ where
\begin{equation}
\hat{h}_{\alpha}(y)=\mathbb{E}_{\alpha}[h(T_{\alpha}y^{\frac{1}{\alpha}})].
\label{h-coagulator}
\end{equation}
and the corresponding inverse local time $\hat{T}_{1}$ has density $\hat{h}_{{\alpha}}(y)f_{\frac{\beta}{\alpha}}(y)$.
\item[(iii)]The distribution of $\tilde{V}|\hat{T}_{1}=y$ is $\mathrm{PK}_{\alpha}(h^{(y)}_{\alpha}\cdot f_{\alpha}),$
where
$$
h^{(y)}_{\alpha}(s)=\frac{h(sy^{\frac{1}{\alpha}})}
{\mathbb{E}_{\alpha}[h(T_{\alpha}y^{\frac{1}{\alpha}})]}.
$$
\end{enumerate}
\end{prop}

\begin{proof}We first recall from~(\ref{subcoag}) that under independent $\mathrm{PD}(\alpha,0)$ and $\mathrm{PD}(
{\beta}/{\alpha},0)$ laws, the bridge $F_{\beta}:=F_{\alpha}\circ G_{\frac{\beta}{\alpha}}$ follows the law of a $\mathrm{PD}(\beta,0)$ bridge with inverse local time at $1,$ $T_{\beta}:=T_{\beta}(1)=T_{\alpha}(T_{
{\beta}/{\alpha}}(1)).$
Hence, under the joint law of $(\tilde{V},Q)$ specified by~(\ref{jointmeasure1}), it follows that, for $F_{V}:=F_{\tilde{V}}\circ G_{Q},$
$$
\mathbb{E}\big[\Omega(F_{\tilde{V}}\circ G_{Q})\big]
=\mathbb{E}\left[ \Omega\big(F_{\alpha}\circ G_{\frac{\beta}{\alpha}}\big)h\big(T_{\alpha}\big(T_{\frac{\beta}{\alpha}}(1)\big)\big)\right]
$$
showing that $F_{V}$ is a $\mathrm{PK}_{\beta}(h\cdot f_{\beta})$ bridge and thus $V\sim \mathrm{PK}_{\beta}(h\cdot f_{\beta})$ in statement~(i). Statements~(ii) and~(iii) follow from straightforward usage of (\ref{jointmeasure2}).
\end{proof}

The next result shows that $(\tilde{V},Q)|T_{V}=r$ is equivalent to $(V_{\alpha},Q_{\frac{\beta}{\alpha}})|T_{\alpha}(T_{\frac{\beta}{\alpha}}(1))=r.$

\begin{prop}\label{propcoag2}For $0<\beta<\alpha<1,$ let $(\tilde{V},Q)$ have a joint distribution, $\mathrm{P}^{\frac{\beta}{\alpha}}_{\alpha}(h),$ specified by~(\ref{jointmeasure1}) or equivalently~(\ref{jointmeasure2}), such that $\tilde{V}\sim \mathrm{PK}_{\alpha}(\tilde{h}_{\frac{\beta}{\alpha}}\cdot f_{\alpha}),$ and 
$(F_{\tilde{V}}, G_{Q})$ are bridges defined in~(\ref{bridges}). Let $V\in \mathcal{P}_{\infty}$ be the ranked masses of the bridge defined by the composition $F_{V}:=F_{\tilde{V}}\circ G_{Q}$, 
with inverse local time at $1$ denoted as $T_{V}$ with density 
$h(t)f_{\beta}(t).$ Then, by a change of variable, the joint distribution $\mathrm{P}^{\frac{\beta}{\alpha}}_{\alpha}(h)$ can be expressed as
\begin{equation}
\int_{0}^{\infty}\left[\int_{0}^{\infty}\mathrm{PD}(\alpha|ry^{-1/\alpha})\mathrm{PD}({\beta}/
{\alpha}|y)
\frac{f_{\alpha}(ry^{-\frac{1}{\alpha}})}{y^{\frac{1}{\alpha}}f_{\beta}(r)}f_{\frac{\beta}{\alpha}}(y)dy\right]h(r)f_{\beta}(r)dr.
\label{jointmeasure3}
\end{equation}
That is, the joint distribution of $(\tilde{V},Q)|T_{V}=r$ is equivalent to $(V_{\alpha},Q_{\frac{\beta}{\alpha}})|T_{\alpha}(T_{\frac{\beta}{\alpha}}(1))=r$, for all $h$, and it is given by
$$
\int_{0}^{\infty}\mathrm{PD}(\alpha|ry^{-1/\alpha})\mathrm{PD}\left(\left.
{\beta}/
{\alpha}\right|y\right)
\frac{f_{\alpha}(ry^{-\frac{1}{\alpha}})}{y^{\frac{1}{\alpha}}f_{\beta}(r)}f_{\frac{\beta}{\alpha}}(y)dy.
$$
\end{prop}

The next Corollary provides an answer to when or under which situation the pair $(Q,V)$ specified by the coagulation $F_{{V}}=F_{\tilde{V}}\circ G_{Q}$ are in the same family of distributions, specifically, distributions of the form $(\mathrm{PK}_{\frac{\beta}{\delta}}(\hat{h}_{\delta}\cdot f_{\frac{\beta}{\delta}}), 0<\beta<\delta<1),$
where $\hat{h}_{\delta}(y)=\mathbb{E}_{\delta}[h(T_{\delta}y^{\frac{1}{\delta}})]$ as in (\ref{h-coagulator}).

\begin{cor}\label{CoagFamily}
For $0<\beta<\sigma<\alpha<1,$ consider the settings in~Proposition~\ref{propcoag} where now $(\tilde{V},Q)\sim \mathrm{P}^{\frac{\beta}{\sigma}}_{\frac{\sigma}{\alpha}}(\hat{h}_{\alpha}),$ and $V\in\mathcal{P}_{\infty}$ is obtained by the coagulation equivalent to ranked masses of $F_{{V}}=F_{\tilde{V}}\circ G_{Q}.$ Then, the variables $(V,\tilde{V},Q)$ have the following marginal (or conditional) distributions
\begin{enumerate}
\item[(i)]$V\sim \mathrm{PK}_{\frac{\beta}{\alpha}}\bigg(\hat{h}_{\alpha}\cdot f_{\frac{\beta}{\alpha}}\bigg)$, for $\hat{h}_{\alpha}(v)=\mathbb{E}_{\alpha}\bigg[h(T_{\alpha}v^{\frac{1}{\alpha}})\bigg]$.
\item[(ii)]$Q~\sim\mathrm{PK}_{\frac{\beta}{\sigma}}\bigg(\hat{h}_{\sigma}\cdot f_{\frac{\beta}{\sigma}}\bigg),$ where $\hat{T}_{1}$ has density $\hat{h}_{\sigma}\cdot f_{\frac{\beta}{\sigma}},$ for
\begin{equation}
\hat{h}_{\sigma}(y)=\mathbb{E}_{\frac{\sigma}{\alpha}}\bigg[\hat{h}_{\alpha}(T_{\frac{\sigma}{\alpha}}y^{\frac{\alpha}{\sigma}})\bigg]=\mathbb{E}_{\sigma}[h(T_{\sigma}y^{\frac{1}{\sigma}})].
\label{h-coagulator2}
\end{equation}
\item[(iii)]The distribution of $\tilde{V}|\hat{T}_{1}=y$ is $\mathrm{PK}_{\frac{\sigma}{\alpha}}\bigg(\big(\hat{h}_{\alpha}\big)^{(y)}_{\frac{\sigma}{\alpha}}\cdot f_{\frac{\sigma}{\alpha}}\bigg),$
where
$$
\big(\hat{h}_{\alpha}\big)^{(y)}_{\frac{\sigma}{\alpha}}(s)
=\mathbb{E}_{\alpha}\bigg[h\big(T_{\alpha}s^{\frac{1}{\alpha}}y^{\frac{1}{\sigma}}\big)\bigg]/\mathbb{E}_{\sigma}[h(T_{\sigma}y^{\frac{1}{\sigma}})].
$$
\end{enumerate}
\end{cor}

\begin{proof}The results follow from Proposition~\ref{propcoag} and manipulating the distributional properties of 
$$
T_{\alpha}\times{\bigg[T_{\frac{\sigma}{\alpha}}\big(T_{\frac{\beta}{\sigma}}(1)\big)\bigg]}^{\frac{1}{\alpha}}\overset{d}=T_{\beta}.
$$
\end{proof}

\begin{rem}In Corollary~\ref{CoagFamily}, $\tilde{V}= \mathrm{FRAG}_{\frac{\sigma}{\alpha},-\frac{\beta}{\alpha}}(V)$ has distribution $\mathrm{PK}_{\frac{\sigma}{\alpha}}\bigg(\widetilde{\big(\hat{h}_{\alpha}\big)}_{\frac{\sigma}{\alpha}}\cdot f_{\frac{\sigma}{\alpha}}),$
where  $\widetilde{\big(\hat{h}_{\alpha}\big)}_{\frac{\sigma}{\alpha}}(s)=\mathbb{E}\bigg[h(T_{\alpha}s^{\frac{1}{\alpha}}T_{\frac{\beta}{\sigma}}^{\frac{1}{\sigma}})\bigg]$.
\end{rem}

\section{Gibbs partitions of $[n]$ derived from $\mathrm{FRAG}_{\alpha,-\beta}$}
Recall from \cite{Pit03,Pit06} that when $V_{\beta}\sim \mathrm{PD}(\beta,0)$, $V_{\beta} |L^{-
{1}/{\beta}}_{1,\beta} = y$ is equivalent in distribution to $V_{\beta} |T_{\beta}=y,$ and has the associated Gibbs partition of $[n]$ described by the $\mathrm{PD}(\beta|y)-\mathrm{EPPF},$
\begin{equation}
p_{\beta}(n_{1},\ldots,n_{k}|y):=\frac{f^{(n-k\beta)}_{\beta,k\beta}(y)}{f_{\beta}(y)} p_{\beta}(n_{1},\ldots,n_{k}),
\label{GibbsalphadeltaEPPF}
\end{equation}
where, as in\cite{HJL,HJL2}, 
$$
\frac{f^{(n-k\beta)}_{\beta,k\beta}(y)}{f_{\beta}(y)} = \mathbb{G}^{(n,k)}_{\beta}(y)\frac{{\beta}^{1-k}\Gamma(n)}{\Gamma(k)},
$$
with, from~\cite{Gnedin06,Pit03,Pit06},
\begin{equation}
\label{bigG}
\mathbb{G}_{\beta}^{(n,k)}(t) =
\frac{\beta^{k}t^{-n}}{\Gamma(n-k\beta)f_{\beta}(t)}
\left[\int_{0}^{t}f_{\beta}(v)(t-v)^{n-k\beta-1}dv\right],
\end{equation}
and $f^{(n-k\beta)}_{\beta,k\beta}(y)$ being the conditional density of $T_{\beta}|K^{[\beta]}_{n}=k$ corresponding to a random variable denoted as $Y^{n-k\beta}_{\beta,k\beta},$ 
as otherwise described in~(\ref{jamesidspecial}) with $\beta$ in place of $\alpha.$ Note, furthermore, as in~\cite{HJL2}, this means $T_{\beta}:=T_{\alpha}(T_{\frac{\beta}{\alpha}}(1))\overset{d}=Y^{(n-K^{[\beta]}_{n}\beta)}_{\beta,K^{[\beta]}_{n}\beta},$ for $K^{[\beta]}_{n}\sim\mathbb{P}^{(n)}_{\beta,0}(k).$ We use these facts to obtain interesting expressions for $\alpha$-Gibbs partitions equivalent to those arising from the $\mathrm{FRAG}_{\alpha,-\beta}$ operator.

\subsection{Gibbs partitions of $[n]$ of $V_{\alpha}|L_{1,\beta},$ equivalently of $\mathrm{FRAG}_{\alpha,-\beta}(V)|L_{1,V}$}\label{sec:41GibbsPartFrag}

Recall from Theorem~\ref{fragtheorem} that the distribution of $V_{\alpha}|L_{1,\beta}=y^{-\beta}$ is equivalent to that of $\tilde{V}=\mathrm{FRAG}_{\alpha,-\beta}(V)|L_{1,V}=y^{-\beta}$, with distribution denoted $\mathrm{PD}_{\alpha|\beta}(\alpha|y):=\mathrm{PK}_{\alpha}\bigg(\omega^{(y)}_{\frac{\beta}{\alpha},\beta}\cdot f_{\alpha}\bigg)$ as in~(\ref{condlocaltime}), where $\omega^{(y)}_{\frac{\beta}{\alpha},\beta}(s)$ is a ratio of stable densities and hence does not have an explicit form for general $0<\beta<\alpha<1.$ We now present results for the EPPF of the  $\mathrm{PD}_{\alpha|\beta}(\alpha|y)$ Gibbs partition of~$[n].$ We first note that since  $T_{\alpha}|K^{[\alpha]}_{n}=k$ is equivalent in distribution to $Y^{n-k\alpha}_{\alpha,k\alpha}$ with density $f^{(n-k\alpha)}_{\alpha,k\alpha},$
the $\mathrm{EPPF}$ can be expressed as
$$
\left[\int_{0}^{\infty}\omega^{(y)}_{\frac{\beta}{\alpha},\beta}(s)f^{n-k\alpha}_{\alpha,k\alpha}(s)ds\right] p_{\alpha}(n_{1},\ldots,n_{k}),
$$
where the first integral term 
is the density of $Y^{(n-k\alpha)}_{\alpha,k\alpha}\times T^{
{1}/{\alpha}}_{
{\beta}/{\alpha}},$ divided by $f_{\beta}(y),$ and does not have an obvious recognizable form. However, we can use the approach in~\cite{HJL} to express $\omega^{(y)}_{\frac{\beta}{\alpha},\beta}$ in terms of Fox-$H$ functions~\cite{Mathai}, leading to an expression for the EPPF in terms of Fox-$H$ functions in the Appendix.

The next result provides a more revealing expression which is not obvious.

\begin{thm}\label{TheoremGibbsparitionFrag}
The $\mathrm{EPPF}$ of the $\mathrm{PD}_{\alpha|\beta}(\alpha|y)$ Gibbs partition of~$[n]$ 
can be expressed as
\begin{equation}
p_{\alpha|\beta}(n_{1},\ldots,n_{k}|y):=\left[\sum_{j=1}^{k}\mathbb{P}^{(k)}_{\frac{\beta}{\alpha},0}(j)\frac{f^{(n-j\beta)}_{\beta,j\beta}(y)}{f_{\beta}(y)} \right]p_{\alpha}(n_{1},\ldots,n_{k}),
\label{Gibbsalpha|delta}
\end{equation}
where $\mathbb{P}^{(k)}_{\frac{\beta}{\alpha},0}(j)=\mathbb{P}_{\frac{\beta}{\alpha},0}(K_{k}=j)$ is the distribution of the number of blocks in a $\mathrm{PD}(\frac{\beta}{\alpha},0)$ partition of $[k]$, and $\sum_{j=1}^{k}\mathbb{P}^{(k)}_{\frac{\beta}{\alpha},0}(j){f^{(n-j\beta)}_{\beta,j\beta}(y)}$ is the conditional density of $T_{\beta}|K^{[\alpha]}_{n}=k$, for the number of blocks $K^{[\alpha]}_{n}$ in a $\mathrm{PD}(\alpha,0)$ partition of $[n],$ with $T_{\beta}:=T_{\alpha}(T_{\frac{\beta}{\alpha}}(1))\overset{d}=
T_{\alpha}\times T^{\frac{1}{\alpha}}_{\frac{\beta}{\alpha}}$ being equivalent to the inverse local time at $1$ of $V_{\beta}\sim \mathrm{PD}(\beta,0)$.
\end{thm}

\begin{proof}
The expression for the $\mathrm{EPPF}$ is the conditional distribution of a $\mathrm{PD}(\alpha,0)$ partition of $[n]$ given $T_{\beta}=y.$ The joint distribution may be expressed as in (\ref{Gibbsalpha|delta}) in terms of the marginal EPPF $p_{\alpha}(n_{1},\ldots,n_{k})$ and the conditional density of $T_{\beta}|K^{[\alpha]}_{n}=k.$ It remains to show that $T_{\beta}|K^{[\alpha]}_{n}=k$ agrees with the expression in (\ref{Gibbsalpha|delta}) as indicated. Recall that $L_{1,\beta}=L_{\frac{\beta}{\alpha}}(L_{1,\alpha})$ and hence the corresponding inverse local time at $1$ is $T_{\beta}:=T_{\beta}(1)=T_{\alpha}(T_{\frac{\beta}{\alpha}}(1))$ corresponding to the coagulation operation dictated by $F_{\beta}=F_{\alpha}\circ G_{\frac{\beta}{\alpha}},$
as expressed in~(\ref{subcoag}).
Sampling from  $F_{\alpha}\circ G_{\frac{\beta}{\alpha}},$ 
that is, according to variables $\bigg(G^{-1}_{\frac{\beta}{\alpha}}(F^{-1}_{\alpha}(U'_{i})), i\in [n]\bigg),$ it follows that this procedure produces a $\mathrm{PD}(\beta,0)$ partition of $[n],$ with $K^{[\beta]}_{n}\overset{d}=K^{[
{\beta}/{\alpha}]}_{K^{[\alpha]}_{n}}$ blocks, where the two components are independent. Furthermore, the order matters, giving $K^{[\alpha]}_{n}$ the interpretation as the number of blocks to be merged, according to a $\mathrm{PD}(\frac{\beta}{\alpha},0)$ partition of $[k],$ for $K^{[\alpha]}_{n}=k\leq n$.  Now from~\cite{HJL2}, $T_{\beta}\overset{d}=Y^{n-K^{[\beta]}_{n}\beta}_{\beta,K^{[\beta]}_{n}\beta}.$ Hence $T_{\beta}|K^{[\alpha]}_{n}=k$ is equivalent to $Y^{(n-K^{[
{\beta}/{\alpha}]}_{k}\beta)}_{\beta,K^{[
{\beta}/{\alpha}]}_{k}\beta},$ which, using ~(\ref{GibbsalphadeltaEPPF}), leads to the description of the density of $T_{\beta}|K^{[\alpha]}_{n}=k$ appearing in~
(\ref{Gibbsalpha|delta}).
\end{proof}

\begin{rem}The result above is equivalent to showing that  $Y^{(n-K^{[
{\beta}/{\alpha}]}_{k}\beta)}
_{\beta,K^{[
{\beta}/{\alpha}]}_{k}\beta}\overset{d}=Y^{(n-k\alpha)}_{\alpha,k\alpha)}\times T^{\frac{1}{\alpha}}_{\frac{\beta}{\alpha}},$ which can be deduced directly using the subordinator representation~\cite[Theorem 2.1 and Proposition 2.1]{HJL2} and decompositions of beta variables. 
\end{rem}

We now describe the distribution of the number of blocks and its limiting behavior.

\begin{cor}\label{condblocks}Consider the $\mathrm{EPPF}$ of a $\mathrm{PD}_{\alpha|\beta}(\alpha|y)$ partition of~$[n]$ as in~(\ref{Gibbsalpha|delta}), for each $y>0.$ Let $\hat{K}_{n}(y)$ denote the corresponding  random number of unique blocks. Then, for $k=1,\ldots,n,$
$$
\mathbb{P}\bigg(K^{[\alpha]}_{n}=k\left|L_{1,\beta}=y^{-\beta}\right.\bigg)
=\mathbb{P}(\hat{K}_{n}(y)=k)=\left[\sum_{j=1}^{k}\mathbb{P}^{(k)}_{\frac{\beta}{\alpha},0}(j)\frac{f^{(n-j\beta)}_{\beta,j\beta}(y)}{f_{\beta}(y)} \right]\mathbb{P}^{(n)}_{\alpha,0}(k),
$$
and, as $n\rightarrow \infty,$ $n^{-\alpha}\hat{K}_{n}(y)\overset{a.s.}\rightarrow \hat{Z}_{\alpha|\beta}(y),$ where $\hat{Z}_{\alpha|\beta}(y)$ is equivalent in distribution to that of $L_{1,\alpha}|L^{-
{1}/{\beta}}_{1,\beta}=y,$ for $L_{1,\beta}=L_{\frac{\beta}{\alpha}}(L_{1,\alpha}).$
\end{cor}

\begin{proof}The distribution follows as a special case of known properties of the distribution of the number of blocks of Gibbs partitions, and is otherwise easy to verify directly from the EPPF. The limiting distribution follows as a special case of~\cite[Proposition 13]{Pit03}.
\end{proof}

\subsection{EPPF of $\mathrm{FRAG}_{\alpha,-\beta}(V)\sim \mathrm{PK}_{\alpha}(\tilde{h}_{\frac{\beta}{\alpha}}\cdot f_{\alpha})$}

Recall from~\cite{Gnedin06,Pit03}, see also~\cite{HJL2}, that if $V\sim \mathrm{PK}_{\beta}(h\cdot f_{\beta})$ with $\mathbb{E}[h(T_{\beta})]=1$, then the $\mathrm{EPPF}$ of its associated Gibbs partition of $[n]$ is described as 
\begin{equation}
p^{[\nu]}_{\beta}(n_{1},\ldots,n_{k})={\Psi}^{[\beta]}_{n,k} \times p_{\beta}(n_{1},\ldots,n_{k}),
\label{alphadeltaEPPF}
\end{equation}
where  $\Psi^{[\beta]}_{n,k}=\mathbb{E}_{\beta}[h(T_{\beta})|K^{[\beta]}_{n}=k]$ and, for clarity, $K^{[\beta]}_{n}$ is the number of blocks of a $\mathrm{PD}(\beta,0)$ partition of $[n].$ 

Theorem~\ref{TheoremGibbsparitionFrag} leads to the EPPF corresponding to $\tilde{V}=\mathrm{FRAG}_{\alpha,-\beta}(V)\sim \mathrm{PK}_{\alpha}\left(\tilde{h}_{\frac{\beta}{\alpha}}\cdot f_{\alpha}\right),$ or any variable in $\mathcal{P}_{\infty}$ having the same distribution.

\begin{prop}\label{PropVnfrag}
Suppose that for $0<\beta<\alpha<1,$ $\tilde{V}\sim \mathrm{PK}_{\alpha}(\tilde{h}_{\frac{\beta}{\alpha}}\cdot f_{\alpha}),$
where 
$
\tilde{h}_{\frac{\beta}{\alpha}}(v):=\mathbb{E}_{\frac{\beta}{\alpha}}\left[h\big(vT^{\frac{1}{\alpha}}_{\frac{\beta}{\alpha}}\big)\right].
$
Then, the $\mathrm{PK}_{\alpha}\left(\tilde{h}_{\frac{\beta}{\alpha}}\cdot f_{\alpha}\right)$ $\mathrm{EPPF}$ of the associated Gibbs partition of $[n]$ can be expressed as
\begin{equation}
\left[\sum_{j=1}^{k}\mathbb{P}^{(k)}_{\frac{\beta}{\alpha},0}(j)\Psi^{[\beta]}_{n,j}\right]p_{\alpha}(n_{1},\ldots,n_{k}),
\label{Gibbsalpha|deltaV}
\end{equation}
and there is the identity, for $T_{\beta}:=T_{\beta}(1)=T_{\alpha}(T_{\frac{\beta}{\alpha}}(1)),$
$$
\mathbb{E}_{\alpha}\left[\tilde{h}_{\frac{\beta}{\alpha}}(T_{\alpha})\left|K^{[\alpha]}_{n}=k\right.\right]
=\mathbb{E}\left[h(T_{\beta})\left|K^{[\alpha]}_{n}=k\right.\right]
=\sum_{j=1}^{k}\mathbb{P}^{(k)}_{\frac{\beta}{\alpha},0}(j)\Psi^{[\beta]}_{n,j}.
$$
\end{prop}

\begin{proof}The $\mathrm{EPPF}$ is equivalent to $\int_{0}^{\infty}p_{\alpha|\beta}(n_{1},\ldots,n_{k}|y)h(y)f_{\beta}(y)dy$, and hence the result follows from~(\ref{Gibbsalpha|delta}) in Theorem~\ref{TheoremGibbsparitionFrag}.
\end{proof}

\begin{rem}
The expression in~(\ref{Gibbsalpha|deltaV}) provides a description of any mass partition with distribution $\mathrm{PK}_{\alpha}\left(\tilde{h}_{\frac{\beta}{\alpha}}\cdot f_{\alpha}\right)$
where 
$
\tilde{h}_{\frac{\beta}{\alpha}}(v):=\mathbb{E}_{\frac{\beta}{\alpha}}\left[h\big(vT^{\frac{1}{\alpha}}_{\frac{\beta}{\alpha}}\big)\right],
$
regardless of whether or not it actually arises from a fragmentation operation.
\end{rem}

As a check, in the case where $(P_{\ell,0})\sim \mathrm{PD}(\beta,\theta),$ (\ref{Gibbsalpha|deltaV}) must satisfy
\begin{equation}
\sum_{j=1}^{k}\mathbb{P}^{(k)}_{\frac{\beta}{\alpha},0}(j)\frac{\Gamma\big(\frac{\theta}{\beta}+j\big)}{\Gamma\big(\frac{\theta}{\beta}+1\big)\Gamma(j)}=
\frac{\Gamma\big(\frac{\theta}{\alpha}+k\big)}{\Gamma\big(\frac{\theta}{\alpha}+1\big)\Gamma(k)}.
\label{Pitmanmoments}
\end{equation}
However,~(\ref{Pitmanmoments}) is verified since it agrees with \cite[exercise 3.2.9, p.66]{Pit06}, with $k$ in place of $n.$
There is the following Corollary in the case of $\beta/\alpha=\frac12$.

\begin{cor}
Specializing Theorem~\ref{PropVnfrag} to the case of $\beta/\alpha=\frac12,$ where $V\sim \mathrm{PK}_{\frac\alpha2}\left(h\cdot f_{\frac{\alpha}{2}}\right),$ and $\Psi^{[
{\alpha}/{2}]}_{n,j}=\mathbb{E}_{
\alpha/2}\left[h(T_{\frac\alpha2})\left|K^{[
{\alpha}/{2}]}_{n}=j\right.\right],$ the $\mathrm{PK}_{\alpha}\left(\tilde{h}_{\frac{1}{2}}\cdot f_{\alpha}\right)$ $\mathrm{EPPF}$ 
in~(\ref{Gibbsalpha|deltaV}) becomes
\begin{equation}
\left[\sum_{j=1}^{k}
{{2k-j-1}\choose{k-1}}2^{j+1-2k}
\Psi^{[
{\alpha}/{2}]}_{n,j}\right]p_{\alpha}(n_{1},\ldots,n_{k}).
\label{GibbsalphahalfV}
\end{equation}
\end{cor}

\subsection{Generating $\mathrm{PD}_{\alpha|\beta}(\alpha|y),$ partitions via $\mathrm{PD}(\alpha,-\beta)$ fragmentation of partitions}\label{sec:43sampling}
While the EPPF's~(\ref{Gibbsalpha|delta}) and~(\ref{Gibbsalpha|deltaV}) are quite interesting from various perspectives, it is not entirely necessary to employ them directly to obtain random partitions from $\mathrm{PD}_{\alpha|\beta}(\alpha|y)$ and  $\mathrm{PK}_{\alpha}(\tilde{h}_{\frac{\beta}{\alpha}}\cdot f_{\alpha}).$ A two-stage sampling scheme may be employed utilizing the dual partition-based interpretation of the $\mathrm{Frag}_{\alpha,-\beta}$ operator. The following scheme can be deduced from Bertoin~\cite{BerFrag}, see also~\cite{Pit99Coag,Pit06}. 
\begin{itemize}
\item[1.]Generate $n$ iid $\mathrm{PD}(\alpha,-\beta)$ partitions of $[n],$ say, $\mathcal{A}_{1},\ldots,\mathcal{A}_{n},$ where, for each $i,$ $\mathcal{A}_{i}:=\{A^{(i)}_{1},\ldots,A^{(i)}_{M^{(i)}_{n}}\}$ with $M^{(i)}_{n}$ blocks.
\item[2.]Independent of this, for each fixed $y,$ generate a $\mathrm{PD}(\beta|y)$ partition of $[n]$, say, $\{C_{1},\ldots,C_{K_{n}(y)}\}$, where $K_{n}(y)$ denotes the number of blocks.
\item[3.] 
For $i=1,\ldots,K_{n}(y)
$, Consider the pairs $(C_{i}, \mathcal{A}_{i})$ and fragment $C_{i}$ by $\mathcal{A}_{i}$ according to
$$
\mathcal{C}_{i}=\left\{C_{i,j}:=C_{i}\cap A^{(i)}_{j}: C_{i}\cap A^{(i)}_{j}\neq \emptyset, j\in\big\{1,\ldots,M^{(i)}_{n}\big\}\right\}.
$$
\item[4.] The collection $\{C_{i,j}\in \mathcal{C}_{i}: i\in [K_{n}(y)]\}$ (arranged according to the least element) constitutes a $\mathrm{PD}_{\alpha|\beta}(\alpha|y)$ partition of $[n],$ with 
\begin{equation}
\hat{K}_{n}(y)\overset{d}=\sum_{i=1}^{K_{n}(y)}K^{(i)}_{|C_{i}|},
\label{countcondfrag}
\end{equation}
where $K^{(i)}_{|C_{i}|}\overset{d}=|\mathcal{C}_{i}|,$ and given $C_{i},$ is equivalent to the number of blocks in a $\mathrm{PD}(\alpha,-\beta)$ partition of $C_{i},$ conditionally independent for $i=1,\ldots, \hat{K}_{n}(y).$
\item[5.] 
Replace Step 2 with a $\mathrm{PK}_{\beta}(h\cdot f_{\beta})$ partition of $[n]$ to obtain a 
 $\mathrm{PK}_{\alpha}\left(\tilde{h}_{\frac{\beta}{\alpha}}\cdot f_{\alpha}\right)$ partition of $[n].$
\end{itemize}

\begin{rem}
The scheme above requires sampling of a $\mathrm{PD}(\beta|y)$ partition of $[n].$ The 
relevant results of~\cite{HJL} show that this is the easiest when $\beta$ is a rational number. In that case, $\mathbb{G}^{(n,k)}_{\beta}(y)$ has a tractable representation in terms of Meijer $G$ functions. So, this applies to, in particular, $(P_{k,1})\sim \mathrm{PD}_{\alpha|\frac{m}{r}}(\alpha|y)$ for every $\alpha>
{m}r,$ where $m<r$ are co-prime positive integers. We look at perhaps the most remarkable case, $\mathrm{PD}_{\alpha|\frac{1}{2}}(\alpha|y),$ in the forthcoming section~\ref{sec:45brownfrag}.
\end{rem}

\subsection{Representations of the $\mathrm{PK}_{\alpha}(\tilde{h}_{\frac{\beta}{\alpha}}\cdot f_{\alpha})$  $\alpha$-diversity, $D$ and $T$ interval partitions, and fixed point equations}\label{sec:44PKhtilde}
The sampling scheme above shows that the number of blocks of 
a $\mathrm{PK}_{\alpha}\left(\tilde{h}_{\frac{\beta}{\alpha}}\cdot f_{\alpha}\right)$, say, $\hat{K}_{n},$ of partition of $[n]$ satisfies
\begin{equation}
\hat{K}_{n}\overset{d}=\sum_{i=1}^{K_{n}}K^{(i)}_{|C_{i}|},
\label{countcondfrag2}
\end{equation}
where $K_{n}$ is the number of blocks in a $\mathrm{PK}_{\beta}(h\cdot f_{\beta})$ partition of $[n]$.
Hence, by standard results for exchangeable partitions,  see for instance~\cite{Pit06}, as $n\rightarrow \infty$,
$(|{C_{i}|/n,i\in [K_{n}]})\overset{d}\rightarrow  ((\tilde{P}_{k,\beta}))$ for $((\tilde{P}_{k,\beta}))$ the size-biased re-arrangement of $P_{\beta}=((P_{k,\beta}))\sim \mathrm{PK}_{\beta}({h}\cdot f_{\beta}),$ and from~\cite{Pit03}, 
$|C_{i}|^{-\alpha}K^{(i)}_{|C_{i}|}\overset{a.s.}\rightarrow Z^{(i)}_{\alpha,-\beta}\sim\mathrm{ML}(\alpha,-\beta).$ Expressed in other terms, as $n\rightarrow\infty,$
\begin{equation}
\left(\frac{|C_{j}|}{n},\frac{K^{(j)}_{|C_{j}|}}{n^{\alpha}}\right)\overset{d}\rightarrow\left(\tilde{P}_{j,\beta},\tilde{P}^{\alpha}_{j,\beta}Z^{(j)}_{\alpha,-\beta}\right)
\label{cdPartition}
\end{equation}
as $j$ varies. The result~(\ref{cdPartition}), and its ranked version involving $P_{\beta}=((P_{j,\beta}))\sim \mathrm{PK}_{\beta}({h}\cdot f_{\beta}),$ can be interpreted as extensions of descriptions in \cite[Theorems 6,7, Propositions 10,11]{AldousT}, see also~\cite[Chapter 9]{Pit06}, to the present setting with general $\beta,$ in place of $\beta=0,$ for the sequence of paired lengths and local times of $D-$ and $T-\mathrm{partition}$ intervals. Hence the next results may be thought of in those terms.

\begin{prop}\label{propdiversityrep}Let $\hat{K}_{n},$ with probability mass function, for $k=1,2,\ldots n,$
$$
\mathbb{P}(\hat{K}_{n}=k)=\mathbb{E}_{\alpha}\left[\left.\tilde{h}_{\frac{\beta}{\alpha}}(T_{\alpha})\right|K^{[\alpha]}_{n}=k\right]\mathbb{P}_{\alpha}^{(n)}(k)=\left[\sum_{j=1}^{k}\mathbb{P}^{(k)}_{\frac{\beta}{\alpha},0}(j)\Psi^{[\beta]}_{n,j}\right]\mathbb{P}_{\alpha}^{(n)}(k),
$$ 
denote the number of blocks in a $\mathrm{PK}_{\alpha}\left(\tilde{h}_{\frac{\beta}{\alpha}}\cdot f_{\alpha}\right)$ partition of $[n],$ as otherwise specified in~(\ref{Gibbsalpha|deltaV}) of
Proposition~\ref{PropVnfrag}. Then, $n^{-\alpha}\hat{K}_{n}\overset{a.s.}\rightarrow \hat{Z},$ where $\hat{Z}$ is the $\alpha$-diversity and has density $
\mathbb{E}_{\frac{\beta}{\alpha}}\left[h(s^{-\frac{1}{\alpha}}T^{\frac{1}{\alpha}}_{\frac{\beta}{\alpha}})\right]
g_{\alpha}(s).$ Furthermore, there is the distributional identity
$$
\hat{Z}\overset{d}=\sum_{k=1}^{\infty}P^{\alpha}_{k,\beta}Z^{(k)}_{\alpha,-\beta}\overset{d}=\sum_{k=1}^{\infty}\tilde{P}^{\alpha}_{k,\beta}Z^{(k)}_{\alpha,-\beta},
$$
where $P_{\beta}=((P_{k,\beta}))\sim \mathrm{PK}_{\beta}({h}\cdot f_{\beta}),$ and  $(\tilde{P}_{k,\beta})$ is its size-biased re-arrangement, independent of $((Z^{(k)}_{\alpha,-\beta}))\overset{iid}\sim \mathrm{ML}(\alpha,-\beta).$
\end{prop}

Recall the (pointwise) identity from size-biased sampling $V_{\alpha,\theta}\sim\mathrm{PD}(\alpha,\theta),$ for $\theta>-\alpha,$ with, for each $\ell=1,2,\ldots,$
\begin{equation}
Z_{\alpha,\theta}=Z_{\alpha,\theta+\ell\alpha}\prod_{i=1}^{\ell}W^{\alpha}_{i}=Z_{\alpha,\theta+\ell}\prod_{i=1}^{\ell}B_{i},
\label{sizebias}
\end{equation}
where $Z_{\alpha,\theta+\ell\alpha}:=T^{-\alpha}_{\alpha,\theta+\ell\alpha}\sim\mathrm{ML}(\alpha,\theta+\ell\alpha)$ is
independent of the independent collection $(W_{1},\ldots,W_{\ell}),$ with each $W_{j}\sim \mathrm{Beta}(\theta+j\alpha,1-\alpha)$, and $Z_{\alpha,\theta+\ell}\sim\mathrm{ML}(\alpha,\theta+\ell)$ is independent of the independent collection $(B_{1},\ldots,B_{\ell}),$ with each $B_{j}\sim \mathrm{Beta}(\frac{\theta+j\alpha}{\alpha},\frac{1-\alpha}{\alpha}).$ See~\cite{DGM,HJL2,PY97}.

Applying the $\mathrm{FRAG}_{\alpha,-\beta}$ operator, in the case of the $\mathrm{PD}$ setting established in~\cite{Pit99Coag}, to $V_{\beta,\ell\alpha-\beta}\sim \mathrm{PD}(\beta,{\ell\alpha-\beta}),$ and $V_{\beta,\ell-\beta}\sim \mathrm{PD}(\beta,{\ell-\beta}),$ for $\ell=1,2,\ldots,$ yields mass partitions $V_{\alpha,\ell\alpha-\beta}\sim \mathrm{PD}(\alpha,{\ell\alpha-\beta}),$ and $V_{\alpha,\ell-\beta}\sim \mathrm{PD}(\alpha,{\ell-\beta}),$ Proposition~\ref{propdiversityrep} and special cases of~(\ref{sizebias}), with $\theta=-\beta,$  leads to the following fixed point equations in the sense of~\cite{AldousBan,GnedinYaku,Iksanovfix}, for certain generalized Mittag-Leffler variables, which we believe are new,  

\begin{prop}\label{fixedpoint}For $0\leq \beta<\alpha<1,$ consider the identity in 
(\ref{sizebias}) for the case where $\theta=-\beta,$ let
$\left(\prod_{i=1}^{\ell}W_{i,k}^\alpha,k\ge 1\right)$ and $\left(\prod_{i=1}^{\ell}B_{i,k},k\ge 1\right)$ denote iid collections of variables having distribution equivalent to $\prod_{i=1}^{\ell}W^{\alpha}_{i}$ and $\prod_{i=1}^{\ell}B_{i},$ respectively, and, independent of other variables, let
$\left(Z^{(k)}_{\alpha,\ell\alpha-\beta},k\ge 1\right)$ and $\left(Z^{(k)}_{\alpha,\ell-\beta},k\ge 1\right)$ denote iid collections of variables with each component having distribution $\mathrm{ML}(\alpha,\ell\alpha-\beta)$ and $\mathrm{ML}(\alpha,\ell-\beta)$, respectively. Then, for each $\ell=1,2,\ldots$, there are the following fixed point equations,
\begin{enumerate}
\item[(i)]for $(\tilde{P}_{k,\beta})\sim \mathrm{GEM}(\beta,{\ell\alpha-\beta}),$ and $Z_{\alpha,\ell\alpha-\beta}\sim\mathrm{ML}(\alpha,\ell\alpha-\beta),$ equivalent in distribution to the $\alpha$-diversity of $(P_{k,\alpha})\sim \mathrm{PD}(\alpha,\alpha\ell-\beta),$ 
$$
Z_{\alpha,\ell\alpha-\beta}\overset{d}=\sum_{k=1}^{\infty}\tilde{P}^
{\alpha}_{k,\beta}\prod_{i=1}^{\ell}W^{\alpha}_{i,k}Z^{(k)}_{\alpha,\ell\alpha-\beta},
$$
and, hence, $\mathbb{E}\left[\sum_{k=1}^{\infty}\tilde{P}^
{\alpha}_{k,\beta}\prod_{i=1}^{\ell}W^{\alpha}_{i,k}\right]=1.$
\item[(ii)]For $(\tilde{P}_{k,\beta})\sim \mathrm{GEM}(\beta,{\ell-\beta}),$  and $Z_{\alpha,\ell-\beta}\sim\mathrm{ML}(\alpha,\ell-\beta),$ equivalent in distribution to the $\alpha$-diversity of $(P_{k,\alpha})\sim \mathrm{PD}(\alpha,\ell-\beta),$ 
$$
Z_{\alpha,\ell-\beta}\overset{d}=\sum_{k=1}^{\infty}\tilde{P}^
{\alpha}_{k,\beta}\prod_{i=1}^{\ell}B_{i,k}Z^{(k)}_{\alpha,\ell-\beta},
$$
and, hence, $\mathbb{E}\left[\sum_{k=1}^{\infty}\tilde{P}^
{\alpha}_{k,\beta}\prod_{i=1}^{\ell}B_{i,k}\right]=1.$ 
\end{enumerate}
\end{prop}
For clarity, for $(\tilde{P}_{k,\beta})$ in~
Proposition~\ref{fixedpoint}, one has $(\tilde{P}_{k,\beta})\overset{d}=(R_{k}\prod_{j=1}^{k-1}(1-R_{j}), k\ge 1),$ where in the case of $(\tilde{P}_{k,\beta})\sim \mathrm{GEM}(\beta,{\ell\alpha-\beta}),$ the $((R_{k}\overset{ind}
\sim \mathrm{Beta}(1-\beta,\ell\alpha+(k-1)\beta))$,
and for $(\tilde{P}_{k,\beta})\sim \mathrm{GEM}(\beta,{\ell-\beta}),$ the $((R_{k}\overset{ind}
\sim \mathrm{Beta}(1-\beta,\ell+(k-1)\beta))$, for all $0\leq \beta<1$. When $\beta=0,$ one has special cases corresponding to the fragmentation results $V_{\alpha,\ell\alpha}=\mathrm{FRAG}_{\alpha,0}(V_{0,\ell\alpha}),$ and $V_{\alpha,\ell}=\mathrm{FRAG}_{\alpha,0}(V_{0,\ell}),$  see~\cite[Chapter 5, p. 119]{Pit06}.
\subsection{$\mathrm{PD}(\alpha,-\frac12)$ Fragmentation of a Brownian excursion partition conditioned on its local time}\label{sec:45brownfrag}
Following Pitman~\cite[Section~8]{Pit03} and~\cite[Section~4.5, p.90]{Pit06}, let $(P_{\ell,0})\sim \mathrm{PD}\big(\frac12,0\big)$ denote the ranked excursion lengths of a standard Brownian motion $B:=(B_{t}:t\in [0,1])$, with corresponding local time at $0$ up till time $1$ given by $L_{1}\overset{d}=\big(2T_{\frac12}\big)^{-\frac12} \overset{d}=|B_{1}|.$ 
Then, it follows that $(P_{\ell,0})|L_{1}=\lambda$ has a $\mathrm{PD}(\frac{1}{2}|\frac{1}{2}\lambda^{-2})$ distribution. Furthermore, with respect to $(P_{\ell}(s))\sim\mathrm{PD}(\frac{1}{2}|\frac{1}{2}s^{-2}),$ we describe the special $\beta=\frac12$ explicit case of the Gibbs partitions (EPPF) of $[n]$ in terms of Hermite functions as derived in~\cite{Pit03}, see also~\cite[Section~4.5]{Pit06}, as
\begin{equation}
p_{\frac12}\left(n_{1},\ldots,n_{k}\left|\frac{1}{2}s^{-2}\right.\right)= s^{k-1}\tilde{H}_{k+1-2n}(s)\frac{\Gamma(n)}{2^{1-n}\Gamma(k)}
p_{\frac{1}{2}}(n_{1},\ldots,n_{k}),
\label{hermiteEPPF}
\end{equation}
where, for $U(a,b,c)$ a confluent hypergeometric function of the
second\break kind~(see~\cite[p.263]{Lebedev72}),
$$
\tilde{H}_{-2q}(s) = 2^{-q} U\left(q,
\frac{1}{2}, \frac{s^2}{2}\right)=\sum_{\ell=0}^{\infty}\frac{{(-s)}^{\ell}}{\ell!}
\frac{\Gamma(q+\frac{\ell}{2})}{2\Gamma(2q)}
2^{q+\frac{\ell}{2}}
$$
is a Hermite function of index $-2q.$ That is, to say 
\begin{equation}
\mathbb{G}^{(n,k)}_{\frac12}\bigg(\frac{1}{2}s^{-2}\bigg)=2^{n-k}s^{k-1}\tilde{H}_{k+1-2n}(s).
\end{equation}

\begin{prop}\label{PropHermitealpha} 
Suppose that $P_{1/2}(s):=(P_{\ell}(s))\sim \mathrm{PD}(\frac{1}{2}|\frac{1}{2}s^{-2}).$ Then, for $\alpha>1/2,$
$$
P_{\alpha}(s)=\mathrm{FRAG}_{\alpha,-\frac{1}{2}}(P_{1/2}(s))\sim \mathrm{PD}_{\alpha|\frac{1}{2}}\left(\alpha\left|\frac{1}{2}s^{-2}\right.\right),
$$
with corresponding $\mathrm{EPPF}$ expressed in terms of a mixture of Hermite functions,
\begin{equation}
\left[\sum_{j=1}^{k}\mathbb{P}^{(k)}_{\frac{1}{2\alpha},0}(j)2^{n-1}s^{j-1}\tilde{H}_{j+1-2n}(s)\frac{\Gamma(n)}{\Gamma(j)}\right]p_{\alpha}(n_{1},\ldots,n_{k}).
\label{MixedalphahermiteEPPF}
\end{equation}
\end{prop}

\begin{proof}The result follows as a special case of Theorem~\ref{fragtheorem} and Theorem~\ref{TheoremGibbsparitionFrag}, and otherwise using the explicit form of the EPPF in~(\ref{hermiteEPPF}). 
\end{proof}

\begin{rem}In order to obtain a partition of $[n]$ corresponding to the EPPF in~(\ref{MixedalphahermiteEPPF}), one can sample from~(\ref{hermiteEPPF}) via the prediction rules indicated in~\cite[eqs.~(111) and (112)]{Pit03}, or otherwise employ the scheme described in Section~\ref{sec:43sampling}.
\end{rem}

Now recall from~\cite[Proposition 14]{Pit03}, equivalently~\cite[Corollary 5]{AldousPit}, that $(\tilde{P}_{\ell}(s)),$ the size biased re-arrangement of $(P_{\ell}(s))\sim\mathrm{PD}(\frac{1}{2}|\frac{1}{2}s^{-2}),$ has  a version with the explicit representation for each $j\ge 1,$ jointly and pointwise,
\begin{equation}
\tilde{P}_{j}(s)=\frac{s^{2}}{s^{2}+S_{j-1}}-\frac{s^{2}}{s^{2}+S_{j}}
\label{browniansize}
\end{equation}
for $S_{j}:=\sum_{i=1}^{j}\tilde{X}_{i}$, where $\tilde{X}_{i}$ are independent and identically distributed variables like $B^{2}_{1}\sim\chi^{2}_{1}$ for $B_{1}\sim \mathrm{N}(0,1)$ a standard Gaussian variable. Now, in view of the results and discussions in Sections~\ref{sec:43sampling} and \ref{sec:44PKhtilde}, we easily obtain the following interesting result.
\begin{prop}Let 
$\hat{K}_{n}(s^{-2}/2)\overset{d}=\sum_{i=1}^{K_{n}(s^{-2}/2)}K^{(i)}_{|C_{i}|}$
denote the number of blocks in a $\mathrm{PD}_{\alpha|\frac{1}{2}}(\alpha|\frac{1}{2}s^{-2})$ partition of $[n],$ as can otherwise be generated according to the scheme in~Section~\ref{sec:43sampling}, by using a $\mathrm{PD}(\alpha,-1/2)$ fragmentation of a $\mathrm{PD}(\frac{1}{2}|\frac{1}{2}s^{-2})$ partition of $[n]$ with blocks $\{C_{i}, i\in [K_{n}(s^{-2}/2)]\}.$ $\hat{K}_{n}(s^{-2}/2)$ is as otherwise described in Corollary~\ref{condblocks} with corresponding $\alpha$-diversity $\hat{Z}_{\alpha|1/2}(s^{-2}/2).$ Then, for $(\tilde{P}_{j}(s),j\ge 1)$ as in~(\ref{browniansize}),
$$
\hat{Z}_{\alpha|1/2}\bigg(\frac{1}{2}s^{-2}\bigg)\overset{d}=\sum_{j=1}^{\infty}[\tilde{P}_{j}(s)]^{\alpha}Z^{(j)}_{\alpha,-1/2}.
$$
Furthermore, as $n\rightarrow\infty,$
$$
\left(\frac{|C_{j}|}{n},\frac{K^{(j)}_{|C_{j}|}}{n^{\alpha}}\right)\overset{d}\rightarrow\left(\tilde{P}_{j}(s),[\tilde{P}_{j}(s)]^{\alpha}Z^{(j)}_{\alpha,-1/2}\right)
$$
jointly for $j=1,2,\ldots,$ where
$\hat{Z}_{\alpha|1/2}(\frac{1}{2}s^{-2})$ is equivalent in distribution to that of $L_{1,\alpha}|L_{1,1/2}=\sqrt{2}s$, and $L_{1,1/2}=L_{\frac{1}{2\alpha}}(L_{1,\alpha}).$
\end{prop}
\subsection{$\mathrm{FRAG}_{\alpha,-\beta}$ for the Mittag-Leffler class}\label{ss:46Mittagsection}
We now present results for an application of the 
$\mathrm{FRAG}_{\alpha,-\beta}$ operator to the most basic case
of the Mittag-Leffler class as described in~\cite{HJL2}, see also~\cite{JamesLamp}. Recall that  for $\lambda>0,$ the Laplace transform of $L_{1,\beta}\overset{d}=T^{-\beta}_{\beta}$  equates with the
Mittag-Leffler function, see for instance~\cite{GorenfloMittag},
expressed as,

$$\mathrm{E}_{\beta,1}(-\lambda)=\mathbb{E}\big[{\mbox
e}^{-\lambda
T^{-\beta}_{\beta}}\big]=\sum_{\ell=0}^{\infty}\frac{{(-\lambda)}^{\ell}}{\Gamma(\beta\ell+1)}.
$$
Let $(N(s),s\ge 0)$ denote a standard Poisson process, where $\mathbb{E}[N(s)]=s,$ and consider the mixed Poisson process $(N(t L_{1,\beta}), t\ge 0).$ Then, as shown in~\cite{HJL2}, for $V_{\beta}\sim \mathrm{PD}(\beta,0),$ $V_{\beta}|N(\lambda L_{1,\beta})=0$ corresponds in distribution to 
\begin{equation}
V_{\beta}(\lambda)\sim\int_{0}^{\infty}\mathrm{PD}(\beta|t)\frac{{\mbox e}^{-\lambda t^{-\beta}}}{\mathrm{E}_{\beta,1}(-\lambda)}f_{\beta}(t)dt
\label{MittagLefflerclass},
\end{equation}
that is, to say a stable Poisson-Kingman distribution with index $\beta$ and $h(t)={\mbox e}^{-\lambda t^{-\beta}}/\mathrm{E}_{\beta,1}(-\lambda).$

\begin{prop}Let $V_{\beta}(\lambda):=(V_{k,\beta}(\lambda))$ have distribution specified in~(\ref{MittagLefflerclass}), and otherwise consider the setting in Theorem~\ref{fragtheorem}. 
\begin{enumerate}
\item[(i)] $\tilde{V}_{\alpha}(\lambda)=\mathrm{FRAG}_{\alpha,-\beta}(V_{\beta}(\lambda))\sim \mathrm{PK}_{\alpha}\left(\tilde{h}_{\frac{\beta}{\alpha}}\cdot f_{\alpha}\right)$ where
$$
\tilde{h}_{\frac{\beta}{\alpha}}(v)
=\frac{\mathrm{E}_{\frac{\beta}{\alpha},1}(-\lambda v^{-\beta})}{\mathrm{E}_{\beta,1}(-\lambda)}.
$$
\item[(ii)] Let $L_{1,\alpha}(\lambda)\overset{d}=T^{-\alpha}_{\alpha}(\lambda)$ denote the corresponding local time/$\alpha$-diversity with density ${\mathrm{E}_{\frac{\beta}{\alpha},1}(-\lambda s^{\frac{\beta}{\alpha}})}g_{\alpha}(s)/{\mathrm{E}_{\beta,1}(-\lambda)},$ then
$$
L_{1,\alpha}(\lambda)\overset{d}=\sum_{j=1}^{\infty}[V_{j,\beta}(\lambda)]^{\alpha} Z^{(j)}_{\alpha,-\beta},
$$
where $((Z^{(j)}_{\alpha,-\beta}))\overset{iid}\sim \mathrm{ML}(\alpha,-\beta),$ independent of $V_{\beta}(\lambda)$.
\end{enumerate}
\end{prop}

\section{Duals in the generalized Mittag-Leffler class}\label{sec:5dualMittagsection}
In Section~\ref{ss:46Mittagsection}, we examined the fragmentation of the Mittag-Leffler variable, $V_{\beta}(\lambda)$ having distribution equivalent to $V_{\beta}|N(\lambda L_{1,\beta})=0.$ The extension to more generalized Mittag-Leffler classes could proceed as in~\cite{HJL2}, by directly conditioning on $V_{\beta}$ as  $V_{\beta}|N(\lambda L_{1,\beta})=j,$ for fixed $j=0,1,2,\ldots.$ Here we show that the corresponding $\mathrm{P}^{\frac{\beta}{\alpha}}_{\alpha}(h)$ joint distribution can arise directly from $(V_{\alpha},Q_{\frac{\beta}{\alpha}})|N(\lambda L_{1,\beta})=j,$ which, has not been previously studied. In order to describe this, we first recall the Laplace transform of the $(P_{\ell})\sim\mathrm{PD}(\beta,\theta)$ local time at $1$ variable, or $\beta$-diversity, say $\hat{L}_{\beta,\theta}\overset{d}=T^{-\beta}_{\beta,\theta}\sim \mathrm{ML}(\beta,\theta),$ with density $g_{\beta,\theta}(s)=s^{\theta/\beta}g_{\beta}(s)/\mathbb{E}[T^{-\theta}_{\beta}],$ as given in \cite[Section~3]{JamesLamp}, and \cite[Section 4.1]{HJL2},

\begin{equation}
\mathbb{E}\big[{\mbox
e}^{-\lambda T^{-\beta}_{\beta,\theta}}\big]=\mathbb{E}\big[{\mbox
e}^{-\lambda^{1/\beta}X_{\beta,\theta}}\big]=\mathrm{E}^{(\frac{\theta}{\beta}+1)}_{\beta,\theta+1}(-\lambda),
\label{MittagLeffler}
\end{equation}
where $X_{\beta,\theta}:=T_{\beta}/T_{\beta,\theta}$ is the Lamperti variable studied in~\cite{JamesLamp}, and
\begin{equation}
\mathrm{E}^{(\frac{\theta}{\beta}+1)}_{\beta,\theta+1}(-\lambda)=\sum_{\ell=0}^{\infty}\frac{{(-\lambda)}^{\ell}}{\ell!}\frac{
\Gamma(\frac{\theta}{\beta}+1+\ell)\Gamma(\theta+1)}
{\Gamma(\frac{\theta}{\beta}+1)\Gamma(\beta
\ell+\theta+1)},\qquad
\theta>-\beta,
\label{altM}
\end{equation}
and from~\cite[Proposition 4.4]{HJL2} there is the density
$$
g^{(0)}_{\beta,\theta+m\beta}(s|\lambda)=
\mathbb{P}(\hat{L}_{\beta,\theta}\in ds|N(\lambda \hat{L}_{\beta,\theta})=j)/ds=\frac{{\mbox e}^{-\lambda s}g_{\beta,\theta+j\beta}(s)}{\mathrm{E}^{(\frac{\theta}{\beta}+j+1)}_{\beta,\theta+j\beta+1}(-\lambda)}.
$$
It follows that for $(P_{\ell})\sim \mathrm{PD}(\beta,\theta),$
 $(P_{\ell})|N(\lambda \hat{L}_{\beta,\theta})=j\sim \mathbb{L}^{(0)}_{\beta,\theta+j\beta}(\lambda),$ such that,
\begin{equation}
\mathbb{L}^{(0)}_{\beta,\theta+j\beta}(\lambda):=\int_{0}^{\infty}\mathrm{PD}(\beta|s^{-\frac1\beta})\,
g^{(0)}_{\beta,\theta+j\beta}(s|\lambda)\,ds,
\label{genMittaglambda2}
\end{equation}
and, hence, in this case,
$$
h(t):=\vartheta^{(\lambda)}_{\beta,\theta+j\beta}(t)=\frac{{\mbox e}^{-\lambda t^{-\beta}}t^{-\theta-j\beta}}{\mathrm{E}^{(\frac{\theta}{\beta}+j+1)}_{\beta,\theta+j\beta+1}(-\lambda)\mathbb{E}[T_{\beta}^{-\theta-j\beta}]}.
$$
Setting $\theta=0,$ it follows for the variables $(V_{\alpha},Q_{\frac{\beta}{\alpha}}),$ and $V_{\beta}$ defined according to coag/frag duality in~\cite{BerPit2000, Pit99Coag}, that the distribution of  
$(V_{\alpha},Q_{\frac{\beta}{\alpha}})|N(\lambda L_{1,\beta})=j\sim\mathrm{P}^{\frac{\beta}{\alpha}}_{\alpha}(\vartheta^{(\lambda)}_{\beta,j\beta}),$ which can be expressed
$$
\int_{0}^{\infty}\int_{0}^{\infty}\mathrm{PD}(\alpha|s)\mathrm{PD}({\beta}/
{\alpha}|y)
\frac{{\mbox e}^{-\lambda s^{-\beta}y^{\frac{-\beta}{\alpha}}}s^{-j\beta}y^{-\frac{j\beta}{\alpha}}}
{\mathrm{E}^{(j+1)}_{\beta,j\beta+1}(-\lambda)\mathbb{E}[T_{\beta}^{-j\beta}]}
f_{\frac{\beta}{\alpha}}(y)f_{\alpha}(s)dyds,
$$
equivalently,
$$
\int_{0}^{\infty}\int_{0}^{\infty}\mathrm{PD}(\alpha|s^{-\frac{1}{\alpha}})\mathrm{PD}({\beta}/
{\alpha}|y^{-\frac{\alpha}{\beta}})
\frac{{\mbox e}^{-\lambda s^{\frac{\beta}{\alpha}}y}}
{\mathrm{E}^{(j+1)}_{\beta,j\beta+1}(-\lambda)}
g_{\frac{\beta}{\alpha},\frac{j\beta}{\alpha}}(y)g_{\alpha,j\beta}(s)dyds.
$$
It follows that
\begin{equation}
\mathbb{E}\big[{\mbox
e}^{-\lambda y^{-\frac{\beta}{\alpha}}T^{-\beta}_{\alpha,j\beta}}\big]=\sum_{\ell=0}^{\infty}\frac{{(-\lambda y^{-\frac{\beta}{\alpha}})}^{\ell}}{\ell!}\frac{
\Gamma(\frac{\beta(j+\ell)}{\alpha})\Gamma(j\beta)}
{\Gamma(\frac{j\beta}{\alpha})\Gamma(j\beta+\ell\beta)}.
\label{partialMittag}
\end{equation}
Hence a direct application of Proposition~\ref{propcoag} leads to identification of all relevant laws as described in the next result. 

\begin{prop}
Consider the specifications in Proposition~\ref{propcoag} with $(\tilde{V},Q)\sim\mathrm{P}^{\frac{\beta}{\alpha}}_{\alpha}(\vartheta^{(\lambda)}_{\beta,j\beta}),$ for each $j=0,1,2\ldots$. Then, in this case,
\begin{enumerate}
\item[(i)]$V\sim \mathbb{L}^{(0)}_{\beta,j\beta}(\lambda)$ as in~(\ref{genMittaglambda2}) with $\theta=0.$
\item[(ii)]$Q$ has the marginal distribution
$$
\int_{0}^{\infty}\mathrm{PD}(\beta/\alpha|y)\frac{\mathbb{E}\big[{\mbox
e}^{-\lambda y^{-\frac{\beta}{\alpha}}T^{-\beta}_{\alpha,j\beta}}\big]}{{\mathrm{E}^{(j+1)}_{\beta,j\beta+1}(-\lambda)}}f_{\frac{\beta}{\alpha},\frac{j\beta}{\alpha}}(y)dy
$$
where $\hat{T}_{1}$ has density ${\mathbb{E}\big[{\mbox
e}^{-\lambda y^{-\frac{\beta}{\alpha}}T^{-\beta}_{\alpha,j\beta}}\big]}f_{\frac{\beta}{\alpha},\frac{j\beta}{\alpha}}(y)/{{\mathrm{E}^{(j+1)}_{\beta,j\beta+1}(-\lambda)}},$ which may be expressed in terms of~(\ref{partialMittag}).
\item[(iii)]$\tilde{V}|\hat{T}_{1}=y,$ has a stable Poisson-Kingman distribution with index $\alpha$ and mixing density
$$
\frac{{\mbox e}^{-\lambda s^{-\beta}y^{\frac{-\beta}{\alpha}}}}
{\mathbb{E}\big[{\mbox
e}^{-\lambda y^{-\frac{\beta}{\alpha}}T^{-\beta}_{\alpha,j\beta}}\big]}f_{\alpha,j\beta}(s).
$$
\item[(iv)]$\tilde{V}=\mathrm{FRAG}_{\alpha,-\beta}(V)$ has a stable Poisson-Kingman distribution with index $\alpha$ and mixing density
$$
\frac{{\mathrm{E}^{(j+1)}_{\frac{\beta}{\alpha},\frac{j\beta}{\alpha}+1}(-\lambda s^{-\beta})}}
{{\mathrm{E}^{(j+1)}_{\beta,j\beta+1}(-\lambda)}}f_{\alpha,j\beta}(s).
$$
\end{enumerate}
\end{prop}

\section{Coagulation and fragmentation of generalized gamma models}\label{sec:6CoagFragGGmodels}
For any $0<\beta<1,$ let $(\tau_{\beta}(s), s\ge 0)$ denote a generalized gamma subordinator specified by its Laplace transform $\mathbb{E}[{\mbox e}^{-w\tau_{\beta}(s)}]={\mbox e}^{-s[(1+w)^{\beta}-1]}.$ 
The generalized gamma subordinator and corresponding mass partitions, and bridges defined by normalization, as described in~\cite{Pit03}, arises in numerous contexts. However, for purpose of this exposition, the reader may refer to its role in the construction of $\mathrm{PD}(\beta,\theta)$ distributions as described in~\cite[Proposition 21]{PY97}. 
More generally, similar to the Mittag-Leffler class, let 
$(N(t T_{\beta}(1)), t\ge 0)$ denote a mixed Poisson process. Then, for $V_{\beta}\sim \mathrm{PD}(\beta,0),$ and for $\zeta>0,$ the distribution of $V_{\beta}|N(\zeta^{\frac{1}{\beta}}T_{\beta}(1)))=m$ corresponds to the laws $
\mathbb{P}^{[m]}_{\beta}(\zeta):= \mathrm{PK}_{\beta}(r^{[m]}_{\beta,\zeta}\cdot f_{\beta})$, where 
\begin{equation}
r^{[m]}_{\beta,\zeta}(t)=\frac{t^{m}{\mbox e}^{-\zeta^{\frac{1}{\beta}}t}}{\mathbb{E}[T^{m}_{\beta}{\mbox e}^{-\zeta^{\frac{1}{\beta}}T_{\beta}}]},
\label{biasPDgen}
\end{equation} 
for $m=0,1,2,\ldots$, as described in~\cite{HJL2,JamesStick,JamesFragblock}. Here, we show how to use Proposition~\ref{propcoag} to easily identify laws and explicit constructions of (dependent)~random measures leading to a coag/frag duality in the case of $m=0,1$ and also show how one may recover the Poisson-Dirichlet coag/frag duality results of~\cite{Pit99Coag}, based on independent $\mathrm{PD}(\alpha,\theta)$ and $\mathrm{PD}(\frac{\beta}{\alpha},\frac{\theta}{\alpha})$ distributions, in the case of $\theta>0,$ using $m=0$, and the general case of $\theta>-\beta,$ using $m=1.$ Results for general $m,$ using Proposition~\ref{propcoag}, are also manageable but require too many additional details for the present exposition.

Now, for a fixed value $\zeta,$ define the scaled subordinator $\hat{T}_{\beta}(\zeta v)=\tau_{\beta}(\zeta v)/{\zeta^{\frac{1}{\beta}}}$, for $0\leq v\leq 1,$ such that
$$
\hat{T}_{1,\beta}(\zeta):=\frac{\tau_{\beta}(\zeta)}{\zeta^{\frac{1}{\beta}}}
$$
has density ${\mbox e}^{-\zeta^{\frac{1}{\beta}}t}{\mbox e}^{\zeta}f_{\beta}(t),$ and one may form a normalized general gamma bridge, for $v\in[0,1],$ as 
\begin{equation}
F_{P_{\beta}(\zeta)}(v)=\frac{\hat{T}_{\beta}(\zeta v)}{\hat{T}_{\beta}(\zeta)}=\frac{\tau_{\beta}(\zeta v)}{\tau_{\beta}(\zeta)}=\sum_{k=1}^{\infty}\hat{P}_{\beta,k}(\zeta)\mathbb{I}_{\{U_{k}\leq v\}},
\end{equation}
where $P_{\beta}(\zeta):=(\hat{P}_{\beta,k}(\zeta))\sim \mathbb{P}_{\beta}(\zeta):=\mathrm{PK}_{\beta}(r_{\beta,\zeta}\cdot f_{\beta})$ with $h(t)=r_{\beta,\zeta}(t):=r^{[0]}_{\beta,\zeta}(t) = {\mbox e}^{-\zeta^{\frac{1}{\beta}}t}{\mbox e}^{\zeta},$ and $\hat{T}_{1,\beta}(\zeta)$ equates to its inverse local time at $1.$ Using this and Proposition~\ref{propcoag},  for $V\overset{d}=P_{\beta}(\zeta)\sim \mathbb{P}_{\beta}(\zeta),$ we obtain versions of $(\tilde{V},Q)\sim \mathrm{P}^{\frac{\beta}{\alpha}}_{\alpha}(r_{\beta,\zeta})$ in this case as follows; for $Q,$
$$
h_{\alpha}(y)=\mathbb{E}_{\alpha}\left[r_{\beta,\zeta}(T_{\alpha}y^\frac{1}{\alpha})\right]=r_{\frac{\beta}{\alpha},\zeta}(y)={\mbox e}^{-\zeta^{\frac{\alpha}{\beta}}y}{\mbox e}^{\zeta},
$$
that is, $Q\overset{d}=P_{\frac{\beta}{\alpha}}(\zeta)\sim \mathbb{P}_{\frac{\beta}{\alpha}}(\zeta)$ with inverse local time at $1,$ $\hat{T}_{1}\overset{d}=\hat{T}_{1,\frac{\beta}{\alpha}}(\zeta)$, with density ${\mbox e}^{-\zeta^{\frac{\alpha}{\beta}}y}{\mbox e}^{\zeta}f_{\frac{\beta}{\alpha}}(y).$ Hence, $\tilde{V}|\hat{T}_{1}=y\sim \mathbb{P}_{\alpha}(\zeta^{\frac{\alpha}{\beta}}y),$ and  has a marginal distribution, corresponding to $\tilde{V}=\mathrm{FRAG}_{\alpha,-\beta}(V),$
$$
\tilde{V}\overset{d}=P_{\alpha}(\tau_{\frac{\beta}{\alpha}}(\zeta))\sim \mathbb{E}[\mathbb{P}_{\alpha}(\tau_{\frac{\beta}{\alpha}}(\zeta))]
:=\int_{0}^{\infty}\mathbb{P}_{\alpha}(\zeta^{\frac{\alpha}{\beta}}y){\mbox e}^{-\zeta^{\frac{\alpha}{\beta}}y}{\mbox e}^{\zeta}f_{\frac{\beta}{\alpha}}(y)dy.
$$
Hence, jointly and component-wise, $(\tilde{V},Q)\overset{d}=\left(P_{\alpha}(\tau_{\frac{\beta}{\alpha}}(\zeta)),P_{\frac{\beta}{\alpha}}(\zeta)\right),$ and $V\overset{d}=
P_{{\beta}}(\zeta)$ is determined by the coagulation 
$$
F_{P_{\beta}(\zeta)}(v)=F_{P_{\alpha}(\tau_{\frac{\beta}{\alpha}}(\zeta))}\left(F_{P_{\frac{\beta}{\alpha}}(\zeta)}(v)\right),
$$
where, for clarity,
$$
F_{P_{\alpha}(\tau_{\frac{\beta}{\alpha}}(\zeta))}(v)=
\frac{\tau_{\beta}\left(\tau_{\frac{\beta}{\alpha}}(\zeta) v\right)}{\tau_{\beta}\left(\tau_{\frac{\beta}{\alpha}}(\zeta)\right)}\qquad {\mbox { and  }}\qquad F_{P_{\frac{\beta}{\alpha}}(\zeta)}(v)=\frac{\tau_{\frac{\beta}{\alpha}}(\zeta v)}{\tau_{\frac{\beta}{\alpha}}(\zeta)}.
$$
Conversely, $\tilde{V}=\mathrm{FRAG}_{\alpha,-\beta}(V)$ is equivalent in distribution to $P_{\alpha}(\tau_{\frac{\beta}{\alpha}}(\zeta))=\mathrm{FRAG}_{\alpha,-\beta}(P_{\beta}(\zeta))\sim \mathbb{E}[\mathbb{P}_{\alpha}(\tau_{\frac{\beta}{\alpha}}(\zeta))].$ Now, following~\cite[Proposition 21]{PY97}, for $\theta>0$, it follows that for $\gamma_{a}\sim\mathrm{Gamma}(a,1),$ $P_{\beta}(\gamma_{\frac{\theta}{\beta}})\sim \mathrm{PD}(\beta,\theta)$, and 
$$
P_{\alpha}(\tau_{\frac{\beta}{\alpha}}(\gamma_{\frac{\theta}{\beta}}))\sim \mathrm{PD}(\alpha,\theta), \quad {\mbox {  independent of  }}\quad P_{\frac{\beta}{\alpha}}(\gamma_{\frac{\theta}{\beta}})\sim \mathrm{PD}\left(\frac{\beta}{\alpha},\frac{\theta}{\alpha}\right),
$$
where also $P_{\alpha}(\tau_{\frac{\beta}{\alpha}}(\gamma_{\frac{\theta}{\beta}}))=\mathrm{FRAG}_{\alpha,-\beta}(P_{\beta}(\gamma_{\frac{\theta}{\beta}}))\sim\mathrm{PD}(\alpha,\theta),$
recovering the coag/frag duality in~\cite{Pit99Coag} for $\theta>0.$ 

\begin{rem}See~\cite{JamesCoagGen} for an earlier, less refined, treatment of these results which requires considerably more effort. 
\end{rem}

\subsection{Results for size biased generalized gamma $\mathbb{P}^{[1]}_{\beta}(\zeta)$}
In order to recover the duality for the entire range of $\theta>-\beta$, we now work with the size biased law of a generalized gamma density. Suppose that $V\overset{d}=P^{[1]}_{\beta}(\zeta)\sim\mathbb{P}^{[1]}_{\beta}(\zeta):= \mathrm{PK}_{\beta}\left(r^{[1]}_{\beta,\zeta}\cdot f_{\beta}\right)$, where
\begin{equation}
r^{[1]}_{\beta,\zeta}(t)={\zeta^{\frac{1}{\beta}-1}}t{\mbox e}^{-\zeta^{\frac{1}{\beta}}t}{\mbox e}^{\zeta}/{\beta},
\label{biasPD}
\end{equation} 
and, now  
$$
\hat{T}^{[1]}_{1,\beta}(\zeta):=\frac{\tau_{\beta}\left(\zeta+\gamma_{\frac{1-\beta}{\beta}}\right)}{\zeta^{\frac{1}{\beta}}},
$$
with density $r^{[1]}_{\beta,\zeta}(t)f_{\beta}(t)$, is the corresponding inverse local time at $1.$ Since this case and a derivation for $\theta>-\beta$ is not well known, we apply Proposition~\ref{propcoag} to identify all the relevant distributions in the next result.
\begin{prop}Consider the variables $V$ and $(\tilde{V},Q),$ forming the coagulation and fragmentation operations as described in Proposition~\ref{propcoag}, where $V\overset{d}=P^{[1]}_{\beta}(\zeta)\sim\mathbb{P}^{[1]}_{\beta}(\zeta):= \mathrm{PK}_{\beta}\left(r^{[1]}_{\beta,\zeta}\cdot f_{\beta}\right)$, and thus $(\tilde{V},Q)\sim \mathrm{P}^{\frac{\beta}{\alpha}}_{\alpha}\left(r^{[1]}_{\beta,\zeta}\right).$ Then,
\begin{enumerate}
\item[(i)]$Q\overset{d}=P^{[1]}_{\frac{\beta}{\alpha}}(\zeta)\sim \mathbb{P}^{[1]}_{\frac{\beta}{\alpha}}(\zeta),$ with $\hat{T}_{1}\overset{d}=T^{[1]}_{1,\frac{\beta}{\alpha}}(\zeta)$
\item[(ii)]$\tilde{V}|\hat{T}_{1}=y\sim \mathbb{P}^{[1]}_{\alpha}(\zeta^{\frac{\alpha}{\beta}}y)$
\item[(iii)]$\tilde{V}\overset{d}=P^{[1]}_{\alpha}\left(\tau_{\frac{\beta}{\alpha}}(\zeta+\gamma_{\frac{\alpha-\beta}{\beta}})\right)$
\item[(iv)]$(\tilde{V},Q)\overset{d}=\left(P^{[1]}_{\alpha}\bigg(\tau_{\frac{\beta}{\alpha}}\big(\zeta+\gamma_{\frac{\alpha-\beta}{\beta}}\big)\bigg),P^{[1]}_{\frac{\beta}{\alpha}}(\zeta)\right)$ jointly and component-wise.
\item[(v)]$P^{[1]}_{\alpha}\bigg(\tau_{\frac{\beta}{\alpha}}\big(\zeta+\gamma_{\frac{\alpha-\beta}{\beta}}\big)\bigg)=\mathrm{FRAG}_{\alpha,-\beta}\left(P^{[1]}_{\beta}(\zeta)\right)$
\item[(vi)]$P^{[1]}_{\beta}(\gamma_{\frac{\theta+\beta}{\beta}})\sim \mathrm{PD}(\beta,\theta)$ for $\theta>-\beta$
\item[(vii)]$P^{[1]}_{\alpha}\left(\tau_{\frac{\beta}{\alpha}}\big(\gamma_{\frac{\theta+\beta}{\beta}}+\gamma_{\frac{\alpha-\beta}{\beta}}\big)\right)\sim \mathrm{PD}(\alpha,\theta)$ independent of $P^{[1]}_{\frac{\beta}{\alpha}}(\gamma_{\frac{\theta+\beta}{\beta}})\sim \mathrm{PD}(\frac{\beta}{\alpha},\frac{\theta}{\alpha})$
\end{enumerate}
\end{prop}
\begin{proof}The results follow from a straightforward application of Proposition~\ref{propcoag} using $h(t)=r^{[1]}_{\beta,\zeta}(t),$ the distributional representation of $\hat{T}_{1}$, and the appropriate Gamma randomization to obtain independent $\mathrm{PD}$ laws. The generalized gamma subordinator representation of  $\hat{T}_{1}$ and Poisson Dirichlet distributional identities can be found in~\cite{HJL2,JamesStick,JamesFragblock}. The independence of the Poisson-Dirichlet laws is due to~\cite[Proposition 21, see p. 877]{PY97} and beta-gamma algebra, see also~\cite[Proposition 2.1]{HJL2}.  
\end{proof}
$P^{[1]}_{\alpha}\bigg(\tau_{\frac{\beta}{\alpha}}\big(\zeta+\gamma_{\frac{\alpha-\beta}{\beta}}\big)\bigg)=\mathrm{FRAG}_{\alpha,-\beta}\left(P^{[1]}_{\beta}(\zeta)\right)$ has distribution
$$
\mathbb{E}\left[\mathbb{P}^{[1]}_{\alpha}\bigg(\tau_{\frac{\beta}{\alpha}}\big(\zeta+\gamma_{\frac{\alpha-\beta}{\beta}}\big)\bigg)\right]
$$
with inverse local time at $1,$ equivalent in distribution to
$$
\hat{T}^{[1]}_{1,\alpha}\bigg(\tau_{\frac{\beta}{\alpha}}\big(\zeta+\gamma_{\frac{\alpha-\beta}{\beta}}\big)\bigg)\overset{d}=
\frac{\tau_{\alpha}\left(\tau_{\frac{\beta}{\alpha}}\big(\zeta+\gamma_{\frac{\alpha-\beta}{\beta}}\big)+\gamma_{\frac{1-\alpha}{\alpha}}\right)}{{[\tau_{\frac{\beta}{\alpha}}\big(\zeta+\gamma_{\frac{\alpha-\beta}{\beta}}\big)]}^{\frac{1}{\alpha}}}.
$$

\begin{rem}\label{remgengam}If $V_{\alpha}\sim \mathrm{PD}(\alpha,0)$ independent of $Q_{\frac{\beta}{\alpha}}\sim \mathrm{PD}(\frac{\beta}{\alpha},0)$, it is evident that $\big(V_{\alpha},Q_{\frac{\beta}{\alpha}}\big)\left|N\left(\zeta^{\frac{1}{\beta}}T_{\alpha}\big(T_{\frac{\beta}{\alpha}}(1)\big)\right)\right.=m\sim \mathrm{P}^{\frac{\beta}{\alpha}}_{\alpha}\left(r^{[m]}_{\beta,\zeta}\right).$
\end{rem}

\begin{appendix}
\section*{}

The $\mathrm{EPPF}$ of the $\mathrm{PD}_{\alpha|\beta}(\alpha|y)$ Gibbs partition of~$[n]$ in Theorem~\ref{TheoremGibbsparitionFrag} may be alternatively expressed in terms of Fox-$H$ functions~\cite{Mathai} as
$$
\frac{\alpha\Hfun{0,2}{2,2}{y}{\left(1-\frac{1}{\beta},\frac{1}{\beta}\right), \left(1-\frac{1}{\alpha}-k,\frac{1}{\alpha}\right)}
{\left(1-\frac{1}{\alpha},\frac{1}{\alpha}\right),(-n,1)}}
{\Hfun{0,1}{1,1}{y}{\left(1-\frac{1}{\beta},\frac{1}{\beta}\right)}
{\left(0,1\right)}}\frac{\Gamma(n)}{\Gamma(k)}p_{\alpha}(n_{1},\ldots,n_{k})
.
$$
The above expression follows by noting the Fox-$H$ representations for $f_{\beta/\alpha}$ and $f^{(n-k\alpha)}_{\alpha,k\alpha},$ followed by applying~\cite[Theorem 4.1]{CarterSpringer}. Otherwise details are similar to the arguments in \cite{HJL}.
\end{appendix}

\begin{funding}
L.F. James was supported in part by grants RGC-GRF 16301521, 16300217 and 601712 of the Research Grants Council (RGC) of the Hong Kong SAR. 
\end{funding}


\end{document}